\documentclass[12pt,a4paper,reqno]{amsart}

\usepackage[T1]{fontenc}
\usepackage{microtype, a4wide, textcomp,fbb}
\usepackage{
	amsmath,  amssymb,  amsthm,   amscd,
	gensymb,  graphicx, etoolbox, 
	booktabs, stackrel, mathtools, mathdots  
}
\usepackage{fullpage}
\usepackage[dvipsnames]{xcolor}
\definecolor{darkblue}{rgb}{0,0,0.7}
\definecolor{darkred}{rgb}{0.7,0,0}
\usepackage[colorlinks=true, linkcolor=darkred, citecolor=darkblue, urlcolor=blue, pagebackref=true, breaklinks=true]{hyperref}
\usepackage[capitalise]{cleveref}
\usepackage{placeins}
\usepackage{relsize}
\usepackage{todonotes}
\usepackage{soul}
\usepackage[inline]{enumitem}
\usepackage[ruled,vlined,linesnumbered]{algorithm2e}
\usepackage{bm}
\usepackage[normalem]{ulem}

\usepackage{tikz}
\usetikzlibrary{arrows}
\usetikzlibrary{shapes}
\usepackage{float}
\usepackage{tabularx}
\usepackage{kantlipsum}
\usepackage{array}
\usepackage[margin=2.6cm]{geometry}
\usepackage{graphicx}

\crefname{algocf}{Algorithm}{Algorithms}
\Crefname{algocf}{Algorithm}{Algorithms}

\newtheorem{proposition}{Proposition}[section]
\newtheorem{lemma}[proposition]{Lemma}
\newtheorem{theorem}[proposition]{Theorem}
\newtheorem{corollary}[proposition]{Corollary}
\newtheorem{question}[proposition]{Question}

\theoremstyle{definition}
\newtheorem{remark}[proposition]{Remark}

\newtheorem{example}[proposition]{Example}
\newtheorem{definition}[proposition]{Definition}

{\hfill $\blacksquare$\par\vspace{0.5em}}
\newtheorem{innercustomthm}{Theorem}
\newenvironment{customthm}[1]
{\renewcommand\theinnercustomthm{#1}\innercustomthm\itshape}
{\endinnercustomthm}

\newcommand{\reg}{{\rm reg}}

\def\H{\mathcal{H}}

\tikzstyle{place}=[draw,circle,minimum size=1mm,inner sep=1pt,outer sep=-1.1pt,fill=black]

\usetikzlibrary{shapes.geometric}
\tikzstyle{places}=[draw,rectangle,minimum size=8pt,inner sep=0pt]
\tikzstyle{placesf}=[draw,rectangle,minimum size=5pt,inner sep=0pt]
\tikzstyle{placec}=[draw,circle,minimum size=8pt,inner sep=0pt]
\tikzstyle{placecf}=[draw,circle, minimum size=5pt,inner sep=0pt]

\def\K{\mathbb{K}}
\def\G{\mathcal{G}}

\def\reg{\mathrm{reg}}
\def\pd{\mathrm{pd}}
\def\Su{\mathrm{Supp}}
\def\dist{\mathrm{dist}}
\def\l{\langle}
\def\r{\rangle}
\def\o{\mathrm{ord}}
\def\lcm{\operatorname{lcm}}

\begin{document}
	
	\title{Barile-Macchia resolutions and the closed neighborhood ideal}
	%\thanks{Last updated: \today}
	
	\author{Ajay P. Joseph}
	\address{National Institute of Technology Karnataka, India}
	\email{ajaymath.217ma001@nitk.edu.in}
	
	\author{Amit Roy}
	\address{Department of Mathematics, SRM University-AP, Amaravati 522240, Andhra Pradesh, India}
	\email{amitiisermohali493@gmail.com; amit.r@srmap.edu.in}
	
	\author{Anurag Singh}
	\address{Indian Institute of Technology Bhilai, India}
	\email{anurags@iitbhilai.ac.in}
	
	
	\vspace*{-0.4cm}
	
	\begin{abstract}
		We investigate the minimal free resolutions of closed neighborhood ideals of graphs within the framework of Barile-Macchia (BM) resolutions. We show that for any tree \(T\), the closed neighborhood ideal \(NI(T)\) is bridge-friendly, and hence its BM resolution is minimal. The combinatorial structure of trees further allows us to construct a maximal critical cell of size \(\alpha(T)\), leading to the equality \(\operatorname{pd}(R/NI(T)) = \alpha(T)\), where \(\alpha(T)\) denotes the independence number of \(T\) and $\operatorname{pd}$ is the projective dimension. Using Betti splitting techniques, we also obtain explicit formulas for the graded Betti numbers of \(NI(P_n)\), where \(P_n\) is the path graph on \(n\) vertices. Finally, we make some observations on the bridge-friendly condition of the closed neighborhood ideals of chordal and bipartite graphs.
	\end{abstract}
	
	\maketitle
	
	\section{Introduction}
	The study of minimal free resolutions of monomial ideals is an active area of research that provides deep connections between combinatorics and commutative algebra. Classical constructions such as the Taylor resolution \cite{taylor1966ideals} and Lyubeznik resolutions \cite{lyubeznik1988new} always yield free resolutions, although they are typically far from minimal. In contrast, the Scarf complex \cite{scarf1969examples} rarely supports a resolution, but whenever it does, the resulting resolution is always minimal. Cellular resolutions provide a unifying framework: a monomial ideal $I \subset R = K[x_1,\dots,x_n]$ is supported on a labeled CW-complex, and consequently, graded Betti numbers, projective dimension, and regularity of the ideal can be extracted from the intrinsic combinatorial structure of the complex. Discrete Morse Theory often offers effective tools for simplifying such complexes; in particular, suitably chosen acyclic matchings on the Taylor complex often produce significantly smaller, and in favorable cases minimal, cellular resolutions \cite{batzies2002discrete, forman2002morse}.
	
	A practical and flexible approach to constructing minimal cellular resolutions originates in the work of Barile-Macchia \cite{biermann_forests}, who introduced a Discrete Morse Theory-based method specifically for edge ideals of forests. Their construction imposes some order on the minimal generators and builds a homogeneous acyclic matching on the Taylor complex compatible with that order; the corresponding critical cells support a cellular resolution, and minimality of the resolution is guaranteed when certain subsets of the generating set satisfy a prescribed ``bridge'' condition. Later, Chau-Kara \cite{TCSK24} generalized this principle to arbitrary monomial ideals by formulating a \emph{bridge-friendly} condition-again expressed as a combinatorial constraint on subsets of the generating set-which characterizes when such Taylor-based matchings yield minimal cellular resolutions in full generality. They referred to the resulting resolution as the \emph{Barile-Macchia resolution} (in short, BM resolution). These methods have since been applied successfully to several classes of ideals, including certain monomial ideals and their Artinian reductions \cite{Chau2025_minimal_cellular_resolutions_five_generators}, edge ideals of forests \cite{biermann_forests}, path ideals \cite{chau2024path}, and certain powers of edge ideals \cite{chau2024powers}.
	
	This article focuses on the BM resolution of the \emph{closed neighborhood ideal} $NI(G)$ of a finite simple graph $G$, with particular attention to the case of trees. The closed neighborhood ideal, introduced by Sharifan and Moradi \cite{SM}, is a square-free monomial ideal whose minimal generators correspond to the closed neighborhoods of the vertices of the graph (see \Cref{known results3}). Although the closed neighborhood ideal is relatively recent, its Stanley-Reisner complex, the dominance complex, and their Alexander dual, the neighborhood complex, have been well studied in the literature \cite{kahle2007random,TakaWaka2025,Ehrenborg}, highlighting the combinatorial significance of these ideals. Considerable work has explored the relationships between algebraic and homological properties of closed neighborhood ideals and the combinatorial structure of the underlying graph (see, e.g., \cite{HW,NQBM,CJRS2025,hien2025associated} and the references therein). To the best of our knowledge, this is the first study that focuses on the minimal free resolution of these ideals, particularly in the framework of the BM resolution.
	
	Among the various classes of simple graphs, trees are of particular interest because of their connections to multiple areas of mathematics and computer science, and because several algebraic invariants of their closed neighborhood ideals can be described in terms of combinatorial invariants of the graph. Our first main result in this setting is the following.
	
	\begin{customthm}{\ref{thm:Bm of trees}}
		Let $T$ be a tree. Then the ideal $NI(T)$ is bridge-friendly. Consequently, the BM resolution of $NI(T)$ is minimal.
	\end{customthm}
	
	\noindent
	To establish this result, we consider $T$ as a rooted tree and define a lexicographic order on the minimal generators of its closed neighborhood ideal, induced by a chosen ordering of the vertices of $T$. We also explore the relationship between closed neighborhood ideals of trees and three well-studied classes of monomial ideals for which the BM resolution is known to be minimal (\Cref{proposition:counter_ex_rooted_hypertree}, \Cref{proposition:NI_Pn_not_generic}, and \Cref{proposition:NI_Pn_not_linearquotient}).

    The projective dimension is a homological invariant of the minimal free resolution of an ideal, measuring its length. By analyzing the critical sets arising from the corresponding acyclic matchings, we are able to establish the following result regarding the projective dimension.
    
	\begin{customthm}{\ref{theorem:pdim_matchingnum_tree}}
		Let $T$ be a tree with the closed neighborhood ideal $NI(T)$ in a polynomial ring $R$. Then the projective dimension
		\[
            \pd(R/NI(T)) = \alpha(T),
		\]
		where $\alpha(T)$ denotes the independence number of $T$.
	\end{customthm}
	
	\noindent
	In fact, \Cref{Algorithm:MaxCritTree} constructs a maximal BM critical set whose size equals the independence number of the tree.
	
	In \Cref{section:betti_path}, we focus on the closed neighborhood ideals of the path graphs $P_n$. By applying the Betti splitting criteria for these ideals, we provide an explicit numerical formula for all graded Betti numbers of $NI(P_n)$ in \Cref{corollary:path_betti}. Moving on, in \Cref{sec:conclusion}, we establish basic results for verifying the bridge, gap, and true gap conditions for certain classes of monomial ideals. Using these results and a correspondence among the critical cells of specific subgraphs of a special class of trees, \Cref{theorem:tree_betti} expresses the graded Betti numbers of their closed neighborhood ideals in terms of the graded Betti numbers of the corresponding subgraphs. Finally, we conclude \Cref{sec:conclusion} with related questions.

	\section{Preliminaries}
	
	In this section, we briefly recall some preliminary notions from graph theory and commutative algebra that are going to be used in the rest of the article.
	
	\subsection{Graph theoretic notions}\label{graph}
	
	%Graphs are useful mathematical structures that finds its applications not only in the realm of mathematics but also in areas like - computer science, network theory, biology, social sciences, linguistics, search engine algorithms, electrical engineering and many more. We define
	A {\it graph} $G=(V(G), E(G))$ is the pair of the vertex set $V(G)$ and the edge set $E(G)$, where $E(G) \subseteq V(G)\times V(G)$ is the set of edges of $G$. For a vertex $x$ of $G$, $G\setminus x$ denotes the graph on the vertex set $V(G)\setminus\{x\}$ with the edge set $\{\{u,v\}\in E(G): x\notin \{u,v\}\}$. The set of {\it neighbors} of $x$, denoted by $N_G(x)$ is the set $\{y \in V(G) : \{x,y\}\in E(G)\}$. On the other hand, the set of \textit{closed neighbors} of $x$ in $G$ is $N_G(x)\cup \{x\}$, and is denoted by $N_G[x]$. The number $|N_G(x)|$ is said to be the {\it degree} of $x$ and is denoted by $\deg(x)$. If $x$ is a vertex of $G$ such that $\deg(x)=1$, then $x$ is called a {\it leaf} of $G$.
	
	The {\it cycle} graph $C_n$ of length $n$ is a graph on the vertex set $\{v_1,\ldots,v_n\}$ with edge set $\{\{v_1,v_n\},\{v_i,v_{i+1}\}: 1\le i\le n-1\}$. Given any arbitrary graph $G$ and a subset $A$ of $V(G)$, the {\it induced subgraph of $G$ on $A$}, denoted by
	$G[A]$, is the graph on the vertex set $V(G[A]) = A$ and the edge set $E(G[A]) = \{e \in E(G) ~|~ e \subseteq A\}$. A subset $S$ of $V(G)$ is called an \textit{independent set} if G[S] contains no edges. The \textit{independence number} of a graph $G$, denoted by $\alpha (G)$, is the maximum cardinality of an independent set in $G$. A connected graph without any induced cycle is called a {\it tree}. A graph on the vertex set $\{x_1,\ldots,x_n\}$ with edge set $\{\{x_i,x_{i+1}\}: 1\le i\le n-1\}$ is called a {\it path graph} of length $n$, and is denoted by $P_n$. One can see that the path graph $P_n$ is an example of a tree. A graph $G$ is called \textit{chordal} if the induced cycles in $G$ are of length at most $3$. Note in particular that trees are connected chordal graphs.
	
	Let $G$ be a connected graph and let $x,y$ be two distinct vertices of $G$. A path of length $n$ between $x$ and $y$ is a collection of $n+1$ vertices $v_1,\ldots,v_{n+1}$ such that $x=v_1, y=v_{n+1}$, and $\{v_i,v_{i+1}\}\in E(G)$ for each $i\in [n]$. We define the {\it distance} of $x$ and $y$ in $G$ to be the number $\dist_G(x,y)$, where
	\[
	\dist_G(x,y)=\min\{n : \text{ there is a path of length }n \text{ between }x \text{ and }y\}.
	\]
	We write $\dist(x,y)$ for $\dist_G(x,y)$ whenever the graph $G$ is clear from the context.
	
	% Given any graph $G$, a subset $M$ of the edge set of $G$ is called a \textit{matching} of $G$ if for all distinct $e,e'\in M$, $e\cap e'=\emptyset$. The maximum size of a matching is called the \textit{matching number} of $G$. In this article, we denote the matching number of a graph $G$ by $a_G$.
	
	\subsection{Minimal graded free resolution}\label{algebra}
	% Let $I$ be an ideal in a polynomial ring $R = K[x_{1}, \ldots, x_{n}]$ over a field $K$. Then every finitely generated $S$-module $M$ can be successively approximated by free modules. Now, each free module is also a graded free module and so it is a direct sum of some shifted copies of $R$ to the power of some integers, i.e., $F_{i} = \bigoplus_{j \in Z} R(-j)^{\beta{i,j}}$. Our aim is to find information about combinatorial objects (graphs in our case) by studying homological invariants and vice versa. We do so by considering the \textit{minimal free resolution} of the underlying module. 
	
	Let $R=\mathbb K[x_1,\ldots,x_n]$ be the polynomial ring in $n$ variables over a field $\K$ equipped with the usual $\mathbb N$-grading. In particular, if $m=\prod_{i=1}^nx_i^{\beta_i}$ is a monomial in $R$ with $\beta_i\in\mathbb N$, then we define $\deg(m)=\sum_{i=1}^n\beta_i$. If $\beta_i\in\{0,1\}$ for all $1\le i\le n$, then the monomial $m$ is said to be a {\it square-free monomial}. An ideal $I$ generated by square-free monomials is called a {\it square-free monomial ideal}. For a square-free monomial ideal $I$ of $R$, a {\it free resolution} of $R/I$ is a long exact sequence of free $R$-modules 
	\[
	\mathcal F_{\cdot}: \,\, 0\rightarrow F_l\xrightarrow{\partial_{l}}\cdots\xrightarrow{\partial_{2}} F_1\xrightarrow{\partial_1} F_0\xrightarrow{\pi} R/I\rightarrow 0, 
	\]
	where $F_0=R$, $\pi$ is the natural quotient map, and $\partial_r$ is the boundary map for all $r\in [l]$. If each of the $F_r$ is $\mathbb N$-graded and $\partial_r$ is a homogeneous map, then $\mathcal F_{\cdot}$ is called a {graded free resolution} of $R/I$. Furthermore, if $\partial(F_r)\subseteq \langle x_1,\ldots,x_n \rangle F_{r-1}$ for all $r\in[l]$, then $\mathcal F_{\cdot}$ is called a {\it minimal graded free resolution} of $R/I$.
	
	Let $\mathcal F.$ be a minimal graded free resolution of $R/I$. Then for each $r>1$, we have $F_r=\bigoplus_{j \in \mathbb{N}}
	R(-d)^{\beta_{r,d}(R/I)}$. The quantity $\beta_{r,d}(R/I)$ is called the {\it $r^{\text{th}}$ graded Betti numbers of $R/I$ in degree $d$}. The projective dimension is a fundamental homological invariant of the graded module $R/I$. It is defined by
\[
\pd(R/I) = \max \{\, r : \beta_{r,d}(R/I) \neq 0 \text{ for some } d \,\},
\]
i.e., the largest homological degree $r$ for which at least one graded Betti number $\beta_{r,d}(R/I)$ is nonzero.
	\subsection{Closed neighborhood ideal}\label{known results3}
	Let $G$ be a graph with the vertex set $V(G)=\{x_1,\ldots,x_n\}$. For each vertex $v$ of $G$, we define the square-free monomial $m(G,v)$ associated with the closed neighborhood of $v$ as $m(G,v)=\prod_{y\in N_G[v]}y$ in the polynomial ring $R=\K[x_1,\ldots,x_n]$. The closed neighborhood ideal $NI(G)$ of $G$ is a square-free monomial ideal defined as follows.
	\[
	NI(G)=\langle \{ m(G,w): w\in V(G) \} \rangle.
	\]
	In this article, our primary aim is to describe an explicit minimal graded free resolution of the closed neighborhood ideal of trees and describe its projective dimension in terms of the combinatorial data associated with the underlying graph. In fact, we show that the recently defined BM resolution is minimal for the closed neighborhood ideal of trees.
	
	\subsection{Barile-Macchia resolution}\label{known results4}
	
	In this subsection, we quickly recall some concepts related to the Barile-Macchia resolution (in short, BM resolution). For a detailed review on this topic, we refer the reader to \cite{TCSK24}. 
	
	Let $I$ be a monomial ideal in a polynomial ring $R$ and let $\G(I)$ denote the set of all minimal generators of $I$ with a total order $>_I$. For $\sigma\subseteq \G(I)$, we denote $\lcm ( \{ m: m\in\sigma \} )$ as $\lcm(\sigma)$. Now, for $\sigma\subseteq \G(I)$ and $m_1\in \G(I)$, we recall the following concepts from \cite{TCSK24}.
	
	\begin{enumerate}
		\item[(1)] The monomial $m_1$ is called a {\it bridge} of $\sigma$ if $m_1\in\sigma$ and $\lcm (\sigma  ) = \lcm ( \sigma\setminus\{m_1\}  )$.
		
		\item[(2)] The monomial $m_1$ is called a {\it gap} of $\sigma$ if $m_1\notin\sigma$ and $\lcm ( \sigma ) = \lcm ( \sigma\sqcup\{m_1\} )$.
		
		\item[(3)] The monomial $m_1$ is called a {\it true gap} of $\sigma$ if $m_1$ is a gap of $\sigma$ and $\sigma\sqcup\{m_1\}$ has no new bridges dominated by $m_1$, i.e., if $m_2\in \G(I)$ such that $m_2$ is a bridge of $\sigma\sqcup\{m_1\}$ and $m_1>_Im_2$, then $m_2$ is a bridge of $\sigma$.
		
		\item[(4)] $\sigma$ is said to be a {\it potentially type-2} element if it has a bridge not dominating any true gaps. 
		
		\item[(5)]  $\sigma$ is said to be a {\it type-2} element if
		\begin{enumerate}
			\item[$(a)$]  $\sigma$ is potentially type-$2$ and
			
			\item[$(b)$] for any potentially type 2 element $\tau(\ne \sigma)$ in $\G(I)$ with $\sigma\setminus\mathrm{sb}(\sigma)=\tau\setminus\mathrm{sb}(\tau)$, we have $\mathrm{sb}(\tau)>\mathrm{sb}(\sigma)$. Here, $\mathrm{sb}(\sigma)=\min_{>}\{m\in\sigma: m\text{ is a bridge of }\sigma\}$ is the smallest bridge of $\sigma$. 
		\end{enumerate}
	\end{enumerate}
	
	In \cite{TCSK24}, Chau-Kara established a sufficient criterion-referred to as the \emph{bridge-friendly condition}-that guarantees the minimality of the BM resolution. For completeness, we recall this condition below.
	
	\begin{definition}\cite[Definition 2.27]{TCSK24}
		A monomial ideal $I$ is said to be \emph{bridge-friendly} if there exists a total ordering $(>_I)$ on $\G(I)$ such that every potentially type-$2$ subset of $\G(I)$ is of type-$2$.
	\end{definition}
	
	By \cite[Theorem 2.29]{TCSK24} if $I$ satisfies the bridge-friendly condition, then $R/I$ has a minimal BM resolution. In this paper, however, we mostly need the following sufficient condition required for the bridge-friendliness property of a monomial ideal $I$. 
	
	\begin{lemma}
		\label{bridge friendly sufficient condition}
		Let $I$ be a monomial ideal with a total order $>$ on the minimal generating set $\G(I)$ of $I$. Suppose the following conditions are satisfied by the elements of $\G(I)$: for any three minimal generators $m_1, m_2,m_3$ of $I$ and two variables $y$ and $z$ in the polynomial ring, either $m_3>m_1$ or $m_3>m_2$ whenever the following two conditions hold:
		\begin{enumerate}
			\item $y\mid m_1$, $y\mid m_3$, but $y\nmid m_2$;
			
			\item $z\mid m_2$, $z\mid m_3$, but $z\nmid m_1$.
		\end{enumerate}
	\end{lemma}
	Then $I$ is bridge-friendly, and hence, the BM resolution of $I$ is minimal.
	\begin{proof}
		Follows immediately from \cite[Lemma 2.9 and Remark 2.12]{chau2024powers}. See also \cite[Theorem 5.4]{CHM20242}.
	\end{proof}
	
	Now, suppose $I$ is a monomial ideal which is also bridge-friendly. Recall from \cite{TCSK24} that the BM resolution of $I$ is obtained from the Taylor complex via a discrete Morse matching. 
    %{\color{blue}In particular, the complex of the BM resolution is a subcomplex of the complex of Taylor resolution, and the differential maps are a restriction of the differential maps in the Taylor resolution.(Required?)} Now, 
	The basis elements of this resolution correspond to certain subsets of $\G(I)$, called critical subsets. In fact, a subset $\sigma \subseteq \G(I)$ is said to be \emph{critical} if it remains unmatched in the Barile-Macchia Morse matching. The following result combinatorially characterizes the critical subsets of $\G(I)$ when the ideal $I$ is bridge-friendly.
	
	\begin{proposition}\cite[Corollary 2.28]{TCSK24}
	\label{proposition:critical_bridgefriendly}
		If a monomial ideal I is bridge-friendly, then the critical subsets of $\mathcal{G}(I)$ are exactly the ones with no bridge and no true gap.
	\end{proposition}
	
	Each critical set $\sigma \subseteq \G(I)$ contributes a free summand to the associated free resolution. Consequently, if $I$ admits a minimal BM resolution, then the graded Betti numbers and the projective dimension of $R/I$ can be described directly in terms of the critical subsets of $\G(I)$, as stated below.

	\begin{proposition}\cite[cf. Corollary 2.4]{TCSK24}
	\label{critical cell betti pd}
		Let $I$ be a monomial ideal in the polynomial ring $R$ with minimal BM resolution. Then
		\begin{align*}
			\beta_{r,d}(R/I)&=|\{\sigma\subseteq \G(I): \sigma \text{ is a critical subset },|\sigma|=r,\deg(\lcm(\sigma))=d \}|,\\
			\pd(R/I)&=\max\{|\sigma|: \sigma \text{ is a critical subset of }\G(I)\}.
		\end{align*}
	\end{proposition}
	
	\section{A minimal Barile-Macchia resolution for trees}
	\label{section:bm_trees}
	In this section, our purpose is threefold. First, we show that the BM resolution of the closed neighborhood ideal $NI(T)$ of a tree $T$ is minimal. Next, we compare the closed neighborhood ideals of trees with the known families of ideals for which the BM resolution is already known to be minimal. Finally, by analyzing the structure of the critical cells, we establish that the projective dimension of the quotient ring coincides with the independence number of the tree.

	Recall that a useful way to verify that the BM resolution is minimal is to show that the corresponding ideal is bridge-friendly. We establish that the ideal $NI(T)$ is bridge-friendly by applying \Cref{bridge friendly sufficient condition}. To apply \Cref{bridge friendly sufficient condition} effectively, we first introduce a linear order on the elements of $\G(I)$. 
	
	For this purpose, we begin by defining an ordering on the vertices of $T$. Recall that any tree can be viewed as a rooted tree once a vertex is chosen as the root. We rename the vertices of $T$ so that the root vertex is denoted by $x_1^{0}$. The \emph{level} of a vertex in $T$ is defined as the distance between that vertex and the root. Hence, the root vertex has level $0$. The vertices at the $i^{\text{th}}$ level are then labeled as $x_1^i, x_2^i, \ldots, x_{p_i}^i$. A typical rooted tree is illustrated in \Cref{fig145} below. From this point onward, $x_1^{0}$ will denote the root of a rooted tree.
	
	\begin{figure}[H]
		\centering
		\begin{tikzpicture}[scale=.55,
			every node/.style={circle, draw, fill=blue!20, inner sep=0.5pt, minimum size=14pt, font=\scriptsize}]
			
			%---------------Vertices with labels inside---------------
			\node (x12) at (0,0)  {$x_1^2$};
			\node (x22) at (2,0)  {$x_2^2$};
			\node (x32) at (4,0)  {$x_3^2$};
			\node (x62) at (6,0)  {\phantom{$x_{p_2}^2$}};
			\node (x82) at (8,0)  {$x_{p_2}^2$}; % If different, adjust label
			
			\node (x11) at (1,2)  {$x_1^1$};
			\node (x21) at (4,2)  {$x_2^1$};
			\node (x71) at (7,2)  {$x_{p_1}^1$};
			
			\node (x10) at (4,4)  {$x_1^0$};
			
			%---------------Edges---------------
			\draw (x12) -- (x11) -- (x22);
			\draw (x62) -- (x71) -- (x82);
			\draw (x10) -- (x21) -- (x32);
			\draw (x11) -- (x10) -- (x71);
			
			%---------------Dots and Vdots---------------
			\node[draw=none, fill=none] at (0,-0.7) {$\vdots$};
			\node[draw=none, fill=none] at (2,-0.7) {$\vdots$};
			\node[draw=none, fill=none] at (4,-0.7) {$\vdots$};
			\node[draw=none, fill=none] at (6,-0.7) {$\vdots$};
			\node[draw=none, fill=none] at (8,-0.7) {$\vdots$};
			\node[draw=none, fill=none] at (5,0) {$\dots$};
			\node[draw=none, fill=none] at (5.4,2) {$\ldots$};
			
		\end{tikzpicture}
		\caption{A rooted tree $T$}
		\label{fig145}
	\end{figure}
	Before defining an ordering of the vertices, we quickly introduce several terminologies that will be used subsequently. Let $T$ be a rooted tree, and let $x$ be any vertex distinct from the root. If the level of $x$ is $i$, then the unique vertex $y \in N_T(x)$ whose level is $i-1$ is called the \emph{parent} of $x$; in this case, $x$ is referred to as a \emph{child} of $y$. Moreover, if $x'$ is another child of $y$, then $x$ and $x'$ are said to be \emph{siblings}. The parent of the parent of $x$ is called the \emph{grandparent} of $x$. Similarly, the child of a child of $x$ is called a \emph{grandchild} of $x$.

	Based on the above relabeling of the vertices, we now define the following order on $V(T)$:
	\[
	x_1^0 > x_1^1 > x_2^1 > \cdots > x_{p_1}^1 > x_1^2 > x_2^2 > \cdots > x_{p_2}^2 > x_1^3 > \cdots .
	\]
	
	Using this ordering, we now define the lexicographic order on the minimal generators of $NI(T)$. In other words, for two minimal monomial generators 
	\[
	m_1 = x_{i_1}^{j_1}x_{i_2}^{j_2}\cdots x_{i_r}^{j_r}
	\quad \text{and} \quad
	m_2 = x_{l_1}^{n_1}x_{l_2}^{n_2}\cdots x_{l_s}^{n_s}
	\]
	of $NI(T)$, with $x_{i_1}^{j_1} > x_{i_2}^{j_2} > \cdots > x_{i_r}^{j_r}$ and $x_{l_1}^{n_1} > x_{l_2}^{n_2} > \cdots > x_{l_s}^{n_s}$ , we say that $m_1 > m_2$ if there exists an integer $1 \le t \le \min\{r,s\}$ such that

	\begin{enumerate}
		\item $(i_p,j_p)=(l_p,n_p)$ for all $p<t$;
		
		\item either $j_t<n_t$, or $j_t=n_t$ and $i_t<l_t$. 
	\end{enumerate}
	
	For a squarefree monomial $m$ in $R = \mathbb{K}[x_1,\ldots,x_n]$, the \emph{support} of $m$, denoted by $\Su(m)$, is the set 
	\[
	\Su(m) = \{x_i : 1 \le i \le n \text{ and } x_i \mid m \}.
	\]
	In the following lemma, for each monomial generator $m$ of the closed neighborhood ideal $NI(T)$, we show that there exists a unique vertex in $\Su(m)$ whose distance from the root vertex is minimum among all vertices in $\Su(m)$. More precisely, we have the following.
	\begin{lemma}
		\label{support lemma}
		Let $T$ be a rooted tree. Let $v\in V(T)$ and $m=m(T,v)\in\G(NI(T))$. Then there exists a unique $w\in\Su(m)$ such that $\dist(x_1^0,w)<\dist(x_1^0,u)$ for all $w\neq u\in\Su(m)$. 
	\end{lemma}
	
	\begin{proof}
		Note that, every vertex $v \neq x_{1}^{0}$ has a unique parent in $T$ and if $w$ is the parent of $v$, $\dist(x_{0}^{1},v) = \dist(x_{0}^{1},w) + 1$. Furthermore, for each child $z$ of $v$, $\dist(x_{0}^{1},z) = \dist(x_{0}^{1},v) + 1$. Therefore, $\dist(x_{0}^{1},w) < \dist(x_{0}^{1},u)$ for all $w \neq u \in \Su ( m(T,v) )$ for $v \neq x_{1}^{0}$. On the other hand, for $v = x_{1}^{0}$, $\dist(x_{0}^{1},v) = \dist(x_{0}^{1},u)-1$ for all $v \neq u \in \Su ( m(T,v) ).$
	\end{proof}
	
	\begin{remark}
		In view of the above lemma, we write a minimal generator of $NI(T)$ as $m=m(T,v,w)$, where $\Su(m)=N_T[v]$ and $w\in \Su(m)$ such that $\dist(x_1^0,w)<\dist(x_1^0,u)$ for all $w\neq u\in\Su(m)$.
	\end{remark}
	In the following lemma, we establish certain relationships among the vertices of the tree $T$ when the supports of two minimal generators of $NI(T)$ intersect.

	\begin{lemma}
		\label{main lemma tree}
		Let $T$ be a rooted tree, and $m_1,m_2\in\G(NI(T))$ with $m_1\neq m_2$ such that $m_1=m(T_1,v_1,w_1)$ and $m_2=m(T,v_2,w_2)$. If $\Su(m_1)\cap\Su(m_2)\neq\emptyset$, then one of the following is true:
		
		\begin{enumerate}
			\item $w_2=v_1$. In this case, $m_1>m_2$.
			
			\item $w_2=w_1$.
			
			\item $v_2 = w_1.$  In this case, $m_2>m_1$.
			
			\item $w_2$ is a child of $v_1$. In this case, $m_1>m_2$.
			
			\item $w_1$ is a child of $v_2$. In this case, $m_2>m_1$.
			
		\end{enumerate}
		
	\end{lemma}
	
	\begin{proof}
		For a non-root vertex $v \in V(T)$,
		\[
		\Su(m(T,v,w)) = \{ w,v,u_{v,i} : 0 \leq i \leq \deg(v)-2 \},
		\] 
		where  $w$ is the parent of $v$ and $u_{v,i}$'s are the children of $v$. For non-root vertices $v_1,v_2 \in V(T)$, if $\Su(m(T,v_1,w_1)) \cap \Su(m(T,v_2,w_2)) \neq \emptyset$, then $N_T[v_1]$ and $N_T[v_2]$ must intersect. Thus, one of the vertices among $w_1$, $v_1$, or a child $u_{v_1,i}$ of $v_1$ must be the same as one of the vertices among $w_2$, $v_2$, or a child $u_{v_2,j}$ of $v_2$. Out of these $9$ possibilities, we observe the following: 
		
		\begin{enumerate}[label=(\roman*)]
			
			\item The case $v_1 = v_2$ is excluded by assumption.
			
			\item If $u_{v_1,i} = u_{v_2,j}$ for some $i$ and $j$, then both $v_1$ and $v_2$ would be parents of the same child, contradicting the acyclicity of the tree. Hence, this case is impossible.
			
			\item $v_1 = u_{v_2,j} $ if and only if $w_1 = v_2$, and likewise $v_2=u_{v_1,i}$ if and only if $v_1 = w_2$.
		\end{enumerate}
		
		Thus, after considering the above possibilities, we obtain the five cases listed above. The corresponding statements about the relation between $m_1$ and $m_2$ can be easily derived from the lexicographic order given to the minimal generators of $NI(T)$. 
		
		Now, suppose one of the vertices is the root vertex. Without loss of generality, assume $v_1$ is the root vertex. Then, 
		\[
		\Su(m(T,v_1,w_1)) = \{ v_1,u_{v_1,i} : 1 \leq i \leq \deg(v_1)-1 \}.
		\]
		For the non-root vertex $v_2$, we have 
		\[ \Su(m(T,v_2,w_2)) = \{ w_2,v_2,u_{v_2,i} : 1 \leq i \leq \deg(v_2)-2 \}. 
		\]
		If $\Su(m(T,v_1,w_1)) \cap \Su( m(T,v_2,w_2)) \neq \emptyset$, then as before, one of the vertices among $v_1$, or a child $u_{v_1,i}$ of $v_1$ must be the same as one of the vertices among $w_2$, $v_2$, or a child $u_{v_2,j}$ of $v_2$. Out of these $6$ possibilities, we observe that:
		\begin{enumerate}[label=(\roman*)]
			
			\item The case $v_1 = v_2$ is excluded by assumption.
			
			\item The cases $v_1 = u_{v_2,j}$ and $u_{v_1,i} = u_{v_2,j}$ cannot occur, since the level of $v_1$ is $0$, the level of $u_{v_1,i}$ is $1$, and the level of $u_{v_2,j}$ is at least $2$.
			
			\item $v_2 = u_{v_1,i} $ if and only if $w_2 = v_1$.
			
		\end{enumerate}
		Hence, as before, all the probable cases lie among the five cases listed above. This concludes the proof of the lemma.
	\end{proof}
	
	We are now ready to prove the main theorem of this section.
	
	\begin{theorem}
	\label{thm:Bm of trees}
		The closed neighborhood ideal of a tree has a minimal Barile-Macchia resolution.
	\end{theorem}
	
	\begin{proof}
		Consider $T$ to be a tree. We proceed to show that the ideal $NI(T)$ with the lexicographic order on the minimal generators as stated before \Cref{support lemma}, satisfies the conditions in \Cref{bridge friendly sufficient condition}. The proof is by contrapositive arguments, i.e., suppose $m_1,m_2,m_3\in\G(NI(T))$, and they satisfy the conditions (1) and (2) in \Cref{bridge friendly sufficient condition}, but $m_1>m_3$ and $m_2>m_3$.
		
		Let $m_1=m(T,v_1,w_1), m_2=m(T,v_2,w_2)$ and $m_3=m(T,v_3,w_3)$. Since $\G(NI(T))$ is totally ordered, without loss of generality, we may assume that $m_1>m_2$. Hence, $v_2$ cannot be the root vertex, and we have $w_2\neq v_2$. Similarly, since $m_2>m_3$, we see that $v_3$ cannot be the root vertex. Moreover, by our assumption, $\Su(m_2)\cap\Su(m_3)\neq\emptyset$. Thus, using \Cref{main lemma tree}, we observe that the following three cases arise:\\
		
		\begin{figure}[H]
			\centering
			
			% -------- CASE I --------
			\begin{minipage}{0.30\textwidth}
				\centering
				\begin{tikzpicture}[scale=.55,every node/.style={circle, draw, fill=blue!20, inner sep=1pt}]
					
					\node (u1) at (0,0) {\phantom{$v_1$}};
					\node (u2) at (2,0) {\phantom{$v_2$}};
					\node (v2) at (1,2) {$v_2$};
					\node[draw=none, fill=none, left=2pt of v2] {$w_3 =$};
					\node (x1) at (1,6) {$x_1^0$};
					\node (w2) at (1,4) {$w_2$};
					\node[draw=none,fill=none] at (1,0) {$\dots$};
					\node[draw=none,fill=none] at (1,5.2) {$\vdots$};
					
					\draw (u1)--(v2)--(u2);
					\draw (v2)--(w2);
					
				\end{tikzpicture}
				
				{\small Case I}
			\end{minipage}
			\hfill
			% -------- CASE II --------
			\begin{minipage}{0.30\textwidth}
				\centering
				\begin{tikzpicture}[scale=.55,every node/.style={circle, draw, fill=blue!20, inner sep=1pt}]
					\node (v1) at (0,0) {\phantom{$v_1$}};
					\node (v2) at (2,0) {\phantom{$v_2$}};
					\node (v3) at (4,0) {\phantom{$v_n$}};
					\node (v4) at (6,0) {\phantom{$v_n$}};
					\node (a) at (1,2) {$v_2$};
					\node (b) at (5,2) {$v_3$};
					\node (c) at (3,4) {$w_2$};
					\node[draw=none, fill=none, right=2pt of c] {$= w_3$};
					\node (d) at (3,6) {$x_1^0$};
					\node[draw=none,fill=none] at (1,0) {$\dots$};
					\node[draw=none,fill=none] at (5,0) {$\dots$};
					\node[draw=none,fill=none] at (3,5.2) {$\vdots$};
					\draw (v1)--(a)--(v2);
					\draw (v3)--(b)--(v4);
					\draw (a)--(c)--(b);
				\end{tikzpicture}
				
				{\small Case II}
			\end{minipage}
			\hfill
			% -------- CASE III --------
			\begin{minipage}{0.30\textwidth}
				\centering
				
				\begin{tikzpicture}[scale=.55,every node/.style={circle, draw, fill=blue!20, inner sep=1pt}]
					
					\node (u1) at (0,0) {\phantom{$w_3$}};
					\node (w3) at (2,0) {$w_3$};
					\node (v2) at (1,2) {$v_2$};
					\node (x1) at (1,6) {$x_1^0$};
					\node (w2) at (1,4) {$w_2$};
					\node[draw=none,fill=none] at (1,0) {$\dots$};
					\node[draw=none,fill=none] at (1,5.2) {$\vdots$};
					
					\draw (u1)--(v2)--(u2);
					\draw (v2)--(w2);
					
				\end{tikzpicture}
				
				{\small Case III}
			\end{minipage}
			
			\caption{The three cases depicting some of the vertices among $v_2,v_3,w_2$ and $w_3$.}
			\label{fig:three_cases}
		\end{figure}
		
		\noindent
		{\bf Case I}: $w_3=v_2$. In this case, the positions of $w_2$ and $w_3$ are as in Case I of \Cref{fig:three_cases}.
		
		Since $\Su(m_1)\cap\Su(m_3)\neq\emptyset$ and $m_1>m_3$, using \Cref{main lemma tree} we see that either $w_3=v_1$, or $w_3=w_1$, or $w_3$ is a child of $v_1$. If $w_3=v_1$, then $v_1=v_2$, a contradiction to the fact that $m_1\neq m_2$. Also, if $w_3=w_1$, then $m_2>m_1$, again a contradiction. Now, if $w_3$ is a child of $v_1$, then we must have $v_1=w_2$, since $T$ is a tree. Thus, $\Su(m_1)\cap \Su(m_3)=\{v_2\}\subseteq \Su(m_2)$. By our assumption, $y=v_2$, and this produces a contradiction to the fact that $y\nmid m_2$.
		
		\noindent
		{\bf Case II}: $w_3=w_2$. In this case, the positions of $v_2, w_2$ and $v_3$ are as in Case II of \Cref{fig:three_cases}. 
		
		Since $m_2>m_3$, we have that $v_2>v_3$ and thus $\Su(m_2)\cap\Su(m_3)=\{w_3\}$. Now, since $m_1>m_3$ and $\Su(m_1)\cap\Su(m_3)\neq\emptyset$, using \Cref{main lemma tree} we see, as before, that either $w_3=v_1$, or $w_3=w_1$, or $w_3$ is a child of $v_1$. In all of these cases, it is easy to see that $\Su(m_2)\cap\Su(m_3)=\{w_3\}\subseteq\Su(m_1)$. By our assumption, $z=w_3$, and this produces a contradiction to the fact that $z\nmid m_1$.
		
		\noindent
		{\bf Case III}: $w_3$ is a child of $v_2$. In this case, the positions of $v_2, w_2$ and $w_3$ are as in Case III of \Cref{fig:three_cases}.
		
		As above, in this case too, we have $\Su(m_2)\cap \Su(m_3)=\{w_3\}$. Since $m_1>m_3$ and $\Su(m_1)\cap\Su(m_3)\neq\emptyset$, by using \Cref{main lemma tree} as in Case II, we again arrive at a contradiction. This completes the proof of the theorem. 
	\end{proof}
	
	\begin{remark}
		
		\begin{figure}[h]
			\centering
			\begin{tikzpicture}
				[every node/.style={circle, draw, fill=blue!20, inner sep=0.3pt}]
				
				\draw (0,0) -- (1,1.5) -- (2,0) -- (0,0);
				\draw (2,0) -- (3,1.5) -- (4,0) -- (2,0);
				\draw (1,1.5) -- (2,3) -- (3,1.5) -- (1,1.5);
				
				\node (x1) at (1,1.5) {$x_1$};
				\node (x2) at (3,1.5) {$x_2$};
				\node (x3) at (2,0) {$x_3$};
				\node (x4) at (4,0) {$x_4$};
				\node (x5) at (2,3) {$x_5$};
				\node (x6) at (0,0) {$x_6$};
			\end{tikzpicture}
			\caption{Bridge Friendly Chordal Graph that does not satisfy condition in \cref{bridge friendly sufficient condition}}
			\label{fig:chordal_ex}
		\end{figure}
		
		It is worthwhile to mention that the condition in \Cref{bridge friendly sufficient condition} is a sufficient but not necessary condition for bridge-friendliness, even for the closed neighborhood ideal of a graph. For instance, if $I$ is the closed neighborhood ideal of the chordal graph drawn in \Cref{fig:chordal_ex}, then $I=\l x_1x_2x_5,x_1x_3x_6,x_2x_3x_4\r$. It is easy to see that if $\sigma$ is a subset of $\G(I)$, then $\sigma$ does not have any bridge and thus no subset of $\G(I)$ is potentially type-2. Consequently, $I$ is bridge-friendly. On the other hand, for the monomial order, $x_1x_2x_5>x_1x_3x_6>x_2x_3x_4$, the condition in \Cref{bridge friendly sufficient condition} is not satisfied. By symmetry, one can check the other possible orderings and also see that the condition in \Cref{bridge friendly sufficient condition} is not satisfied.

		% \begin{figure}[h]
			% 	\centering
			
			% 	\begin{minipage}{0.47\textwidth}
				% 		\centering
				% 		\begin{tikzpicture}
					% 			[every node/.style={circle, draw, fill=blue!20, inner sep=0.5pt}]
					
					% 			\draw (0,0) -- (1,1.5) -- (2,0) -- (0,0);
					% 			\draw (2,0) -- (3,1.5) -- (4,0) -- (2,0);
					% 			\draw (1,1.5) -- (2,3) -- (3,1.5) -- (1,1.5);
					
					% 			\node (x1) at (1,1.5) {$x_1$};
					% 			\node (x2) at (3,1.5) {$x_2$};
					% 			\node (x3) at (2,0) {$x_3$};
					% 			\node (x4) at (4,0) {$x_4$};
					% 			\node (x5) at (2,3) {$x_5$};
					% 			\node (x6) at (0,0) {$x_6$};
					% 		\end{tikzpicture}
				
				% 		\caption{}
				% 		\label{fig:chordal_ex}
				% 	\end{minipage}
			% 	\hfill
			% 	\begin{minipage}{0.47\textwidth}
				% 		\centering
				% 		\begin{tikzpicture}[scale=0.85,
					% 			every node/.style={circle, draw, fill=blue!20, inner sep=1pt, minimum size=15pt}]
					
					% 			\node (y1) at (1.2,2) {$y_1$};
					% 			\node (y2) at (2.8,2) {$y_2$};
					% 			\node (yt) at (4.6,2) {$y_t$};
					% 			\node (w1) at (1.2,4) {$w_1$};
					% 			\node (w2) at (2.8,4) {$w_2$};
					% 			\node (wt) at (4.6,4) {$w_t$};
					% 			\node (x) at (3,6) {$x$};
					
					% 			\draw (x) -- (w1) -- (y1);
					% 			\draw (x) -- (w2) -- (y2);
					% 			\draw (x) -- (wt) -- (yt);
					% 			\draw[loosely dashed] (w2) -- (wt);
					% 			\draw[loosely dashed] (y2) -- (yt);
					% 			\draw[loosely dashed] (y1) -- +(0,-1.5);
					% 			\draw[loosely dashed] (y2) -- +(0,-1.5);
					% 			\draw[loosely dashed] (yt) -- +(0,-1.5);
					% 		\end{tikzpicture}
				
				% 		\caption{Caption}
				% 	\end{minipage}
			
			% \end{figure}
		The example in \Cref{fig:chordal_ex} shows that the idea used in \Cref{thm:Bm of trees} may not work for the closed neighborhood ideals of general chordal graphs.
	\end{remark}
	
	\subsection{Comparison with known families of ideals:} 
	\label{subsection:known_bm_families}
	In the literature, several classes of monomial ideals are known to admit a minimal (generalised) BM resolution. Among them, the edge ideals of rooted hypertrees, generic monomial ideals, and monomial ideals with linear quotients form broad families that satisfy this minimality property. Our goal in this subsection is to show that the closed neighborhood ideals of trees do not belong to any of these classes.
	
	Let $\mathcal H$ be a hypergraph on the vertex set $V(\mathcal H)$ with edge set $E(\mathcal H)$, and let $I( \mathcal{H} )= \left\langle \{ \mathbf{x}_{ \mathcal{E} }: \mathcal{E} \in E ( \mathcal{H} ) \} \right\rangle$ denote its edge ideal. Here, $\mathbf{x}_{\mathcal{E}}=\prod_{v \in \mathcal{E}} x_{v}$. A {\it host graph} $H$ of the hypergraph $\mathcal H$ is a graph on the same vertex set $V(\mathcal H)$ such that, for each edge $\mathcal E\in E(\mathcal H)$, the induced subgraph $H[\mathcal E]$ of $H$ is connected. A hypergraph is called a {\it hypertree} if it admits a tree as its host graph. Furthermore, a hypertree is called a {\it rooted hypertree} if there exists a fixed vertex $x$ in $\mathcal H$ and a host tree $H$ such that every edge $\mathcal E$ of $\mathcal H$ consists of vertices lying at distinct distances from $x$ in $H$ (see \cite[Section 5]{CHM20242} for more details). In this case, we say that $x$ is the {\it root vertex} of $H$. It was shown in \cite[Theorem 5.4]{CHM20242} that if $\mathcal H$ is a rooted hypertree, then $I(\mathcal H)$ is bridge-friendly, and hence has a minimal BM resolution.  
	
	In the following proposition, we exhibit an infinite family of trees whose closed neighborhood ideals are not the edge ideals of any rooted hypertree.

	\begin{proposition}
		\label{proposition:counter_ex_rooted_hypertree}
		Let $T$ be a tree containing a vertex $x$ such that $\deg(x)\ge 5$, $\deg(w)=2$ for each $w\in N_T(x)$, and $\mathrm{dist}(x,a)\ge 3$ for any leaf vertex $a$ of $T$. Then $NI(T)$ is not the edge ideal of any rooted hypertree.
	\end{proposition}
	
	\begin{figure}[h]
		\centering
		\begin{tikzpicture}[scale=0.75,
			every node/.style={circle, draw, fill=blue!20, inner sep=1pt, minimum size=15pt}]
			
			\node (y1) at (1.2,2) {$y_1$};
			\node (y2) at (2.8,2) {$y_2$};
			\node (yt) at (4.6,2) {$y_t$};
			\node (w1) at (1.2,4) {$w_1$};
			\node (w2) at (2.8,4) {$w_2$};
			\node (wt) at (4.6,4) {$w_t$};
			\node (x) at (3,6) {$x$};
			\node[draw=none,fill=none] at (3.7,4) {$\dots$};
			\node[draw=none,fill=none] at (3.7,2) {$\dots$};
			\node[draw=none,fill=none] at (1.2,1.4) {$\vdots$};
			\node[draw=none,fill=none] at (2.8,1.4) {$\vdots$};
			\node[draw=none,fill=none] at (4.6,1.4) {$\vdots$};
			
			\draw (x) -- (w1) -- (y1);
			\draw (x) -- (w2) -- (y2);
			\draw (x) -- (wt) -- (yt);
		\end{tikzpicture}		
		\caption{An example of a tree in \Cref{proposition:counter_ex_rooted_hypertree}}
	\end{figure}
	
	\begin{proof}
		The proof is by contradiction. Suppose $NI(T)$ is the edge ideal of a rooted hypertree $\H$ and let $T'$ be the host tree of $\H$. Let $N_T(x)=\{w_1,w_2,\ldots, w_t\}$ for some $t\ge 5$, and $N_T(w_i)=\{x,y_i\}$. From the given hypothesis, we have $\left\{x\cdot\prod_{i=1}^tw_i,xw_iy_i\colon i\in[t]\right\}\subseteq\G(NI(T))$. Suppose $z\in V(T)$ is the root vertex of the host tree $T'$. Let $\dist(z,x)=d\ge 0$. Since the induced subgraph $T'[x,w_i,y_i]$ is connected and also a tree for each $i\in [t]$, we have $\dist(z,w_i)\in\{d-2,d-1,d+1,d+2\}$. Now, $t\ge 5$ implies $\dist(x,w_i)=\dist(x,w_j)$ for some $i\neq j$, a contradiction to the fact that $\{x,w_1,\ldots,w_t\}\in E(\H)$ and $\H$ is a rooted hypertree. Thus, $NI(T)$ is not the edge ideal of any rooted hypertree. 
	\end{proof}
	
	Next, we turn to the notion of generic monomial ideals. For a monomial \(m\), define 
	\[
	\o_{x_i}(m)=\max\{k \in \mathbb{Z} : x_i^k \mid m \},
	\]
	and for a monomial ideal \(I\), define 
	\[
	\o_{x_i}(I) = \min \{ \o_{x_i}(m) : m \in \mathcal{G}(I) \}.
	\]
	Based on these definitions, we say that a monomial ideal \(I\) is \emph{generic} if the following condition holds: whenever 
	\(\o_{x_i}(m)=\o_{x_i}(m') > \o_{x_i}(I)\) 
	for some variable \(x_i\) and monomials \(m, m' \in \mathcal{G}(I)\), there exists a generator 
	\(m'' \in \mathcal{G}(I)\) 
	such that 
	\(m'' \mid \lcm(m,m')\). 
	In other words, no two minimal generators attain the same strictly positive excess exponent in a variable without the existence of a third generator dividing their least common multiple.  
	
	Recently, in \cite[Theorem~3.1]{CHM20242}, Chau-Ha-Maithani showed that \(R/I\) admits a minimal (generalized) BM resolution whenever \(I\) is a generic monomial ideal. In the following proposition, we present an infinite family of trees whose closed neighborhood ideals are not generic monomial ideals.

	\begin{proposition}
		\label{proposition:NI_Pn_not_generic}
		For each $n\ge 5$, the closed neighborhood ideal $NI(P_n)$ of the path graph $P_n$ is not a generic monomial ideal.
	\end{proposition}
	\begin{proof}
		We have $\G(NI(P_n))=\{x_1x_2,x_2x_3x_4,\ldots,x_{n-3}x_{n-2}x_{n-1},x_{n-1}x_n\}$. Observe that $\o_{x_2}(NI(P_n))=0$ since $n\ge 5$. On the other hand, $\o_{x_2}(x_1x_2)=\o_{x_2}(x_2x_3x_4)=1$, but there is no $m\in \G(NI(P_n))\setminus\{x_1x_2,x_2x_3x_4\}$ such that $m$ divides $x_1x_2x_3x_4=\lcm(x_1x_2,x_2x_3x_4)$.
	\end{proof}
	
	It was shown in \cite[Theorem~3.3]{CHM20242} that if a monomial ideal \(I\) has linear quotients, then \(R/I\) admits a minimal (generalized) BM resolution. Recall that a monomial ideal \(I\) is said to have \emph{linear quotients} if there exists an ordering \(m_1,m_2,\ldots,m_r\) of \(\mathcal{G}(I)\) such that, for each \(1 < i \le r\), the colon ideal \((m_1,\ldots,m_{i-1}) : m_i\) is generated by variables. In the following proposition, we present an infinite family of trees whose closed neighborhood ideals do not have linear quotients.

	\begin{proposition}
		\label{proposition:NI_Pn_not_linearquotient}
		For each $n\ge 6$, the closed neighborhood ideal $NI(P_n)$ of the path graph $P_n$ does not have a linear quotient.
	\end{proposition}
	\begin{proof}
		By \cite[Theorem 3.4]{CJRS2025}, $\reg(NI(P_n))=\lfloor\frac{n}{2}\rfloor+1\ge 4$ since $n\ge 6$. On the other hand, if $NI(P_n)$ has linear quotient, then by \cite[Corollary 8.2.14, Theorem 8.2.15]{HHBook}, we should have $\reg(NI(P_n))=3$, which is a contradiction. 
	\end{proof}
	
	\subsection{Projective dimension from minimal BM resolution:}\label{pd tree section} The fact that the closed neighborhood ideal of a tree admits a minimal BM resolution has several important combinatorial implications. In particular, one can analyze the critical cells in the corresponding acyclic matching to obtain information about the minimal resolution of the ideal. In this subsection, we present such an application. More precisely, we analyze the critical cells to show that the projective dimension of the ideal can be expressed in terms of the independence number of the underlying tree. We begin with some notations that will be useful in the sequel.

	Let $T$ be a rooted tree. For $v \in V(T)$, define 
	\[
	M_n (v,T) = \{ m(T,u) \in \mathcal{G}(NI(T)) : \dist_T (u,v) = n \}.
	\]
	
	and 
	\[
	M_{\leq n} (v,T) = \{ m(T,u) \in \mathcal{G}(NI(T)) : \dist_T (u,v) \leq n \}.
	\]
	For paths $P_{u,v}$ from $u$ to $v$ and $P_{w,v}$ from $w$ to $v$, define the relation $\sim_v$ by
	\[
	P_{u,v} \sim_v P_{w,v} \text{ if and only if } | V(P_{u,v}) \cap V(P_{w,v}) | \geq 2.
	\]
	It is easy to see that \( \sim_v \) defines an equivalence relation on the set \( \{ P_{u,v} : m(T,u) \in M_{\leq 2}(v,T) \cap \sigma ;\ u \neq v \} \), where \( \sigma \subseteq \mathcal{G}(NI(T)) \). The corresponding set of equivalence classes is denoted by \( P_{2,\sigma}(v,T) \). The following proposition characterizes the existence of a bridge or a gap in $\sigma$ in terms of the number of elements of the equivalence class $P_{2,\sigma}(v,T)$.
	
	%It is easy to observe that $\sim_v$ is reflexive and symmetric. For $\sigma \subseteq \mathcal{G}(NI(T))$, $\sim_v$ is transitive in the set $\{ P_{u,v} : m(T,u) \in M_{\leq 2}(v,T) \cap \sigma \}$ and the set of equivalence classes is denoted by $P_{2,\sigma}(v,T)$.

	\begin{proposition}
		\label{proposition : gap_bridge_characterization}
		Let $T$ be a rooted tree and $\sigma \subseteq \mathcal{G}(NI(T))$. Then $m(T,v)$ is a bridge (respectively, gap) of $\sigma$ if and only if:
		\begin{enumerate}[label = (\roman*)]
			\item $m(T,v) \in \sigma \ ( \text{respectively, } m(T,v) \notin \sigma),$
			\item $| P_{2,\sigma} (v,T) | = \deg_T (v),$
			\item $M_1 (v,T) \cap \sigma \neq \emptyset.$
		\end{enumerate}
	\end{proposition}
	
	\begin{proof}
		Observe that $\lcm (\sigma \setminus \{ m(T,v) \}) = \lcm (\sigma \cup \{ m(T,v) \})$ if and only if $x_i \mid \lcm (\sigma \setminus m(T,v))$ for all $x_i \in \mathrm{Supp}(m(T,v)).$ The proof follows from the two characterizations: $(i) \ | P_{2,\sigma} (v,T) | = \deg_T (v) $ if and only if $x_i \mid \lcm (\sigma \setminus m(T,v))$ for all $x_i \in N_T (v)$ and $(ii) \ M_1 (v,T) \cap \sigma \neq \emptyset $ if and only if $v \mid \lcm (\sigma \setminus m(T,v)).$ 
	\end{proof}

Using the above proposition we first prove the following interesting result.

\begin{proposition}
\label{proposition:no_bridge_ind_num}
    For a rooted tree $T$, if $\sigma \subseteq \mathcal{G}(NI(T))$ contains no bridges, then $|\sigma| \leq \alpha_{T}$.
\end{proposition}
\begin{proof}
Let $V_{\sigma}=\{v\in V(G)\mid m(T,v)\in\sigma\}$ and $\sigma_E=\{e\in E(G)\mid e\subseteq V_{\sigma}\}$. We show that $|\sigma| \leq \alpha_{T}$ by induction on $|\sigma_E|$. If $|\sigma_E|=0$, then the statement is clearly true. Now, suppose $|\sigma_E|=k$ for a positive integer $k$, and the statement holds for any $\widetilde{\sigma}\subseteq \G(NI(T))$ containing no bridges such that $|\widetilde{\sigma}_E|<k$. Our aim now is to construct a set $\sigma'$ such that $\sigma'\subseteq\G(NI(T))$, $|\sigma'|=|\sigma|$, $\sigma'$ has no bridges, and $|\sigma'_E|<k$. Observe that if we can construct such a $\sigma'$, then the proof is complete by the induction hypothesis.

Choose an edge $\{u,v\}\in\sigma_E$ such that $\mathrm{level}(v)$ is maximum among all the vertices in the edges of $\sigma_E$. Therefore, no children of $v$ is in $V_{\sigma}$. Observe that since $m(T,v)$ is not a bridge of $\sigma$, by \cref{proposition : gap_bridge_characterization}, there exists a child, say $x_v$, of $v$ such that $x_v\nmid\lcm(\sigma\setminus \{m(T,v)\})$. Observe that $N_T[x_v]\cap V_{\sigma}=\{v\}$. We now replace $v$ by a vertex $w$ according to the following rule: 
 \[
        w = \begin{cases}
            x_v  & \text{ if } x_v \text{ has no leaf neighbor}, \\
             y & \text{ if } y \text{ is a leaf neighbor of }x_v.
        \end{cases}
    \]
Note that if $x_v$ has several leaf neighbors, then we fix one of them and name it as $y$. Observe that if $x$ has no leaf neighbor, then $m(T,x)\in\G(NI(T))$. Thus, if we take $\sigma' = (\sigma \setminus \{ m(T,v) \}) \cup \{ m(T,w) \}$, then $\sigma'\subseteq \G(NI(T))$. Moreover, the choice of $w$ ensures that $w$ has no neighbors in $V_{\sigma'}=(V_{\sigma} \setminus \{v\})\cup\{w\}$. Thus, $\sigma_E'$ is obtained from $\sigma_E$ by deleting the edge $(u,v)$. Consequently, $|\sigma'_E| = |\sigma_E| - 1$.

It remains to show that $\sigma'$ has no bridges. Indeed, by construction $w$ has no neighbors in $V_{\sigma'}$, and hence $m(T,w)$ cannot be a bridge of $\sigma'$. Now, suppose there exists $z \in V_{\sigma} \setminus \{ v \}$ such that $m(T,z)$ is a bridge in $\sigma'$. By \cref{proposition : gap_bridge_characterization}, we have $\dist(w, z) \in \{1, 2\}$. 

 \begin{itemize}
        \item If $\dist(w, z) = 1$, then $(w, z)$ is an edge. Since $w$ is chosen to be non-adjacent to any vertex in $\sigma \setminus \{v\}$, this is impossible.
        \item If $\dist(w, z) = 2$, then $z$ must be either a grandparent or a grandchild or a sibling of $w$. Note that if $z$ is a grandchild or a sibling of $w$, then $\mathrm{level}(z)>\mathrm{level}(v)$, a contradiction. On the other hand, if $z$ is a grandparent of $w$, then since $v \notin V_{\sigma'}$, the only possibility is that $w = x_v$, in which case $z = u$.
 Then, observe that $N_T[z]\cap N_T[w]=\{v\}$, and hence $m(T,u)$ is a bridge of $\sigma$, again a contradiction.
             \end{itemize}

Thus, $\sigma'$ has no bridges. By the induction hypothesis, $|\sigma'| \leq \alpha_{T}$, and since $|\sigma| = |\sigma'|$, the result follows.
\end{proof}

	Next, we present an algorithm for constructing a maximal BM critical set of a rooted tree. The main idea is to build the set step by step so that the vertices corresponding to the minimal generators form a maximal independent set. The algorithm works by checking each vertex in the tree and deciding whether to include its closed neighborhood in the critical set, while ensuring that no two selected vertices are connected. This process continues until no more vertices can be added without breaking this condition.
	
	\begin{algorithm}[H]
		\caption{Construction of Maximal BM-Critical Set of a Rooted Tree}
		\label{Algorithm:MaxCritTree}
		\SetAlgoLined
		
		\textbf{Initialization:} 
		$\sigma = \emptyset$ ;
		$V' = \{ v \in V(T) : m(T,v) \in \mathcal{G}(NI(T)) \}$ ;
		$V_{\sigma} = \emptyset$
		
		\While{there exists $v \in V'$ such that $\deg_T(v) = 1$}
		{
			\For{each $v \in V'$ with $\deg_T(v) = 1$}
			{
				Add $m(T,v)$ to $\sigma$\;
				Add $v$ to $V_{\sigma}$\;
				Remove $v$ from $V'$\;
			}
		}
		\While{$V' \neq \emptyset$}
		{
			Choose $v$ of $V'$ having the highest level\;
			\If{ $m(T,u) \notin \sigma$ for all children $u$ of $v$}
			{
				Add $m(T,v)$ to $\sigma$\;
				Add $v$ to $V_{\sigma}$\;
			}
			Remove $v$ from $V'$\;
		}
		\Return $\sigma, V_{\sigma}$
	\end{algorithm}
	
	\begin{example}
		For the tree \( T \) shown in \Cref{figure:AlgorithmMaximalBMCritical}, the vertices highlighted in red correspond to those obtained by applying Algorithm~\ref{Algorithm:MaxCritTree}, i.e., $V_{\sigma}=\{x_2,x_3,x_7,x_9,x_{10},x_{11},x_{12}\}$. Moreover, $\sigma=\{x_1x_2x_4x_5,x_1x_3x_6,x_4x_7,x_5x_9,x_6x_{10},x_6x_{11},x_8x_{12}\}$.
		
		\begin{figure}[h!]
			\centering
			\begin{tikzpicture}[
				scale=0.7, 
				every node/.style={transform shape}, 
				blueNode/.style={circle, draw, fill=blue!20, inner sep=0.8pt, minimum size=18pt},
				redNode/.style={circle, draw, fill=red!20, inner sep=0.8pt, minimum size=18pt}
				]
				
				% Red nodes (were blue)
				\node[redNode] (x7) at (0,0) {$x_7$};
				\node[redNode] (x9) at (4,0) {$x_9$};
				\node[redNode] (x10) at (6,0) {$x_{10}$};
				\node[redNode] (x11) at (8,0) {$x_{11}$};
				\node[redNode] (x12) at (2,-1.5) {$x_{12}$};
				\node[redNode] (x2) at (2,3) {$x_2$};
				\node[redNode] (x3) at (7,3) {$x_3$};
				
				% Blue nodes (were red)
				\node[blueNode] (x8) at (2,0) {$x_8$};
				\node[blueNode] (x4) at (1,1.5) {$x_4$};
				\node[blueNode] (x5) at (4,1.5) {$x_5$};
				\node[blueNode] (x6) at (7,1.5) {$x_6$};
				\node[blueNode] (x1) at (4.5,4.5) {$x_1$};
				
				% Edges
				\draw (x7) -- (x4) -- (x8) -- (x12);
				\draw (x4) -- (x2) -- (x5) -- (x9);
				\draw (x2) -- (x1) -- (x3) -- (x6);
				\draw (x10) -- (x6) -- (x11);
				
			\end{tikzpicture}
			\caption{The vertices in red color forms $V_{\sigma}$ as an output of Algorithm \ref{Algorithm:MaxCritTree}}
			\label{figure:AlgorithmMaximalBMCritical}
		\end{figure}
	\end{example}
	
	In the following lemma, we derive some important properties of $\sigma$ and $V_{\sigma}$ obtained from Algorithm~\ref{Algorithm:MaxCritTree}.
	
	\begin{lemma}
	\label{lemma:sigma_properties}
		Let $T$ be a rooted tree. Then, for $\sigma, V_{\sigma}$ obtained by Algorithm \ref{Algorithm:MaxCritTree};
		\begin{enumerate}[label = (\roman*)]
			\item $V_{\sigma}$  is a maximal independent set of $T$,
			
			\item $\sigma$ has no bridges,
			
			\item For all $m(T,v) \in \sigma,\ | P_{2,\sigma} (v,T) | \geq \deg(v) - 1$. Furthermore, if an equivalence class of paths is missing, it must be the one passing through the parent of $v$,
			
			\item $\sigma$ has no true gaps,
			
			% \item {\color{blue} \st{$\lcm (\{ m : m \in \sigma \}) = \prod\limits_{v \in V(T)} v$.}}
		\end{enumerate}
	\end{lemma}
	
	\begin{proof}
		\begin{enumerate}[label = (\roman*)]
			\item Suppose $V_{\sigma}$ is not an independent set. Then there exists $ u,v  \in V_{\sigma}$ such that $u$ is the parent of $v$. Thus, we may assume that $u$ is not a leaf vertex. Then, by Algorithm \ref{Algorithm:MaxCritTree} we have $m(T,u) \notin \sigma$ since $v\in V_{\sigma}$, and $v$ is a child of $u$, again a contradiction. Hence, $V_{\sigma}$ is an independent set. Moreover, if $w \notin V_{\sigma}$ is not a neighbor of a leaf vertex, by Algorithm \ref{Algorithm:MaxCritTree} there exists a child $u$ of $w$ such that $m(T,u) \in \sigma$, implying $u \in V_{\sigma}$. Hence, $V_{\sigma}$ is a maximal independent set of $T$.
			
			%Assume $ u,v  \in V_{\sigma}$ and that $u$ is the parent of $v$. If a leaf vertex $u$ is the root vertex , then $v \notin V_{\sigma}$, because $m(T,v) \notin \mathcal{G}(NI(T))$. Otherwise, $v \in V_{\sigma}$ implies $m(T,v) \in \sigma$. By Algorithm \ref{Algorithm:MaxCritTree} we have $m(T,u) \notin \sigma$, consequently $u \notin V_{\sigma}$. Thus, $V_{\sigma}$ is an independent set. Moreover, if $v \notin V_{\sigma}$ is not a neighbor of a leaf vertex, by Algorithm \ref{Algorithm:MaxCritTree} there exists a child $u$ of $v$ such that, $m(T,u) \in \sigma$ implying $u \in V_{\sigma}$. Hence, $V_{\sigma}$ is a maximal independent set of $T$.
			
			\item If $m(T,v)$ is a bridge of $\sigma$, by \cref{proposition : gap_bridge_characterization} (iii), there exists $m(T,u) \in \sigma$ such that $u \in N_T (V)$. This implies, $u \in V_{\sigma}$. This contradicts the fact that $V_{\sigma}$ is an independent set.
			
			\item Let $m(T,v) \in \sigma$. Then, for every child $u$ of $v$, we have $m(T,u) \notin \sigma$. Consequently, by \cref{Algorithm:MaxCritTree}, for each such $u$, there exists a child $w_u$ of $u$ such that $m(T,w_u) \in \sigma$. The paths $P_{w_u,v}$ belong to distinct equivalence classes in $P_{2,\sigma}(v,T)$, since each $w_u$ has a different parent. Therefore, $|P_{2,\sigma}(v,T)| \geq \deg(v) - 1.$
			
			\item Note that if $v$ is a leaf vertex of $T$, then $m(T,v)$ cannot be a gap of $\sigma$. Now, suppose $m(T,v)\in \G(NI(T))$ is a gap of $\sigma$ such that $v$ is not a leaf vertex of $T$. Then, by Algorithm \ref{Algorithm:MaxCritTree}, there exists a child $u$ of $v$ such that $m(T,u) \in \sigma$. Since $|P_{2,\sigma}(u,T)| \geq \deg(u) - 1$, it follows that $|P_{2,\sigma \cup \{ m(T,v) \}}(u,T)| = \deg(u)$. As $\sigma$ has no bridges, by \cref{proposition : gap_bridge_characterization}, $m(T,u)$ is a new bridge of $\sigma \cup m(T,v)$. Since $m(T,u) < m(T,v)$, $m(T,v)$ is not a true gap of $\sigma$. Consequently, $\sigma$ has no true gaps.
			
			%If $m(T,v)$ is a gap of $\sigma$, there exists a child $u$ of $v$ such that, $m(T,u) \in \sigma$. Since $|P_{2,\sigma}(u,T)| \geq \deg(u) - 1$, it follows that $|P_{2,\sigma \cup \{ m(T,v) \}}(u,T)| = \deg(u)$. As $\sigma$ has no bridges, by \cref{proposition : gap_bridge_characterization}, $m(T,u)$ is a new bridge of $\sigma \cup m(T,v)$. Since $m(T,u) < m(T,v)$, $m(T,v)$ is not a true gap of $\sigma$.
			
			% \item {\color{blue} \st{Let $v\in V(T)$. If $v$ is a leaf vertex or adjacent to a leaf vertex, then by Algorithm} \ref{Algorithm:MaxCritTree}\st{, it is easy to see that $v \mid \lcm (\{ m : m \in \sigma \})$. If $v$ is not adjacent to a leaf vertex, then again by Algorithm} \ref{Algorithm:MaxCritTree}\st{, either $m(T,v)\in\sigma$ or there exists a child $u$ of $v$ such that $m(T,u) \in \sigma$. Thus, $\lcm (\{ m : m \in \sigma \}) = \prod\limits_{v \in V(T)} v$, as desired.}}
			
			%        Assume $v \in V(T)$ and $v \nmid \lcm (\{ m : m \in \sigma \}).$ It follows that $m(T,v) \notin \sigma.$ If $v$ is adjacent to a leaf vertex $u$, then $m(T,u) \in \sigma$ and hence $v \mid \lcm (\{ m : m \in \sigma \}) $, since $v \mid m(T,u).$ If $v$ is not adjacent to a leaf vertex, then by Algorithm \ref{Algorithm:MaxCritTree}, there exists a child $u$ of $v$ such that $m(T,u) \in \sigma.$ Again, $v \mid \lcm (\{ m : m \in \sigma \}) $ since $v \mid m(T,u).$
			
		\end{enumerate}
		
	\end{proof}
	
	We now prove that, for a tree $T$, the cardinality of the set $V_\sigma$ obtained from \Cref{Algorithm:MaxCritTree} is equal to the independence number of $T$.

	\begin{proposition}
	\label{proposition:IndNo_Vsigma}
		Given a rooted tree $T$, we have $\alpha(T) = |V_{\sigma}|,$ where $\alpha(T)$ is the independence number of $T$.
	\end{proposition}
	
	\begin{proof}
		Let $A$ be an arbitrary independent set of $T$. Define $A' = \{ x \in A : \deg_T (y) \neq 1 \text{ for all } y \in N_T (x) \} \cup \{ y \in N_T (x) : x \in A;\ \deg_T (y) = 1;\ y < y' \text{ for all leaf siblings } y' \text{ of } y \}.$ Observe that $A'$ is an independent set, and $|A| = |A'|.$ Now, define the function $\psi : A' \to V_{\sigma}$ by,
		\[
		\psi (v) =
		\begin{cases}
			v, & \text{if } v \in V_{\sigma}, \\
			\min_{<} \{ u \in V_{\sigma} : u \text{ is a child of } v \}, & \text{if } v \notin V_{\sigma}.
		\end{cases}
		\]
		
		Note that if $v \notin V_{\sigma}$, then by Algorithm~\ref{Algorithm:MaxCritTree}, the vertex $v$ has at least one child in $V_{\sigma}$. Hence, the function $\psi$ is well-defined. To prove that $\psi$ is injective, we divide the proof into the following four cases.

		\noindent
		{\bf Case I:} $v_1,v_2 \in V_\sigma$. 
		
		$\psi (v_1) = \psi (v_2)$ implies $v_1 = v_2.$
		
		\noindent
		{\bf Case II:} $v_1, v_2 \notin V_{\sigma}$.
		
		If $\psi (v_1) = \psi (v_2)$, then there exists a child $u$ of $v_1$ and $v_2$. Since a tree is acyclic, we must have $v_1=v_2$.
		
		\noindent
		{\bf Case III:} $v_1 \in V_{\sigma} \text{ and } v_2 \notin V_{\sigma}$.
		
		If $\psi (v_1) = \psi (v_2)$, then $v_1 = \psi (v_2)$, consequently $v_1$ is the child of $v_2$, contradicting the assumption that $A'$ is an independent set.
		
		\noindent
		{\bf Case IV:} $v_1 \notin V_{\sigma} \text{ and } v_2 \in V_{\sigma}$.
		
		$\psi (v_1) = \psi (v_2)$ implies $\psi (v_1) = v_2$. This says that $v_2$ is a child of $v_1$, again a contradiction to the fact that $A'$ is an independent set.
		
		Therefore, $\psi$ is an injective function. It follows that, $|A| = |A'| \leq |V_{\sigma}|.$ As $A$ is an arbitrary independent set, we have $\alpha(T) \leq |V_{\sigma}|.$ Moreover, by \Cref{lemma:sigma_properties} (i), $V_{\sigma}$ itself is an independent set. Thus, $\alpha (T) = | V_{\sigma} |$, as desired.
	\end{proof}
    
    We are now ready to prove the main theorem of this subsection.
    
    \begin{theorem}
    \label{theorem:pdim_matchingnum_tree}
        For any tree $T$, $\pd(R/NI(T)) = \alpha (T)$, where $\alpha (T)$ is the independence number of $T$.
    \end{theorem}

    \begin{proof}
        It follows from \Cref{thm:Bm of trees} that $NI(T)$ is bridge-friendly. Combining \Cref{proposition:critical_bridgefriendly} with \Cref{proposition:no_bridge_ind_num}, we see that any critical cell $\tau \subseteq \mathcal{G}(NI(T))$ satisfies $|\tau| \leq \alpha_{T}$. Consequently, by \Cref{critical cell betti pd}, we have $\pd(R/NI(T)) \leq \alpha_{T}$. Moreover, \Cref{proposition:IndNo_Vsigma} shows that $\alpha_{T} = |V_\sigma|= |\sigma|$ for $\sigma$ obtained in \Cref{Algorithm:MaxCritTree}. Note that, by \Cref{lemma:sigma_properties} $(ii)$ and $(iv)$ we have that $\sigma$ is critical. Hence,
        \[
            |\sigma| \leq \pd(R/NI(T)) \leq \alpha_{T} = |\sigma|.
        \]
        Therefore, $\pd(R/NI(T)) = \alpha(T)$.
    \end{proof}

	\section{Betti numbers for path graphs}
	\label{section:betti_path}
	
	In this section, we derive an explicit formula for the Betti numbers of the closed neighborhood ideal of path graphs. By \Cref{thm:Bm of trees} and \Cref{critical cell betti pd}, computing the Betti numbers \(\beta_{r,d}\) reduces to counting the critical cells $\sigma$ of size \(r\) where the degree of the $\mathrm{lcm}$ of the elements in \(\sigma\) is \(d\). A convenient approach to determining the number of such critical cells is to relate them to the critical cells of paths with fewer vertices. However, as we discuss in \Cref{sec:conclusion}, establishing such recursive relationships becomes significantly more challenging for trees in which the grandparents of leaf vertices do not have any leaf neighbors. For path graphs, we overcome this difficulty by using the fact that their closed neighborhood ideals admit a Betti splitting. The concept of Betti splitting, introduced by Francisco-H\`a-Van Tuyl in \cite{FHVT}, is defined as follows.

	\begin{definition}\cite[Definition 1.1]{FHVT}
		\label{definition:betti_splitting}
		Let $I,J$ and $K$ be monomial ideals such that the minimal generating set of $I$ can be expressed as $\mathcal{G}(I) = \mathcal{G}(J) \sqcup \mathcal{G}(K).$ Then, the decomposition $I = J + K$ is called a \emph{Betti splitting} if, for all integers $r \geq 0,\ d \geq 0$, the following equality holds:
		\begin{equation}
			\label{eq:betti_splitting}
			\beta_{r,d}(I) = \beta_{r,d}(J) + \beta_{r,d}(K) + \beta_{r-1,d}(J \cap K).
		\end{equation}
	\end{definition}
	
	The following theorem provides a useful way to find a Betti splitting of an ideal.
	
	\begin{theorem}\cite[Corollary 2.7]{FHVT}
		\label{theorem:betti_splitting}
		Let $I \subseteq R=\mathbb{K}[x_1,x_2,\dots,x_n]$ be a monomial ideal in $R$. Fix a variable $x_i$ in $R$, and define two monomial ideals $J$ and $K$ as follows:
		\[
		\mathcal{G}(J) = \{ m \in \mathcal{G}(I) : x_i \mid m \} \text{ and } \mathcal{G}(K) = \{ m \in \mathcal{G}(I) : x_i \nmid m \}.
		\]
		Then, $I=J+K$ is a Betti splitting if $J$ has a linear resolution.
	\end{theorem}
	
	A closely related ideal to the closed neighborhood ideal of path graphs is the $3$-path ideal of path graphs. Recall that if $P_n$ is the path graph on $n$ vertices $\{x_1,\ldots,x_n\}$ with edge set $E(P_n)=\{\{x_i,x_{i+1}\}: i\in [n-1]\}$, then the set of minimal generators of the $3$-path ideal $J_3(P_n)$ is $\{x_i x_{i+1} x_{i+2} : i \in [n-2]\}$. A formula for the Betti numbers of $J_3(P_n)$ can be deduced from \cite[Theorem~4.14]{AlFa2018}. For our purposes, however, we employ the following simplified formula.
	
	\begin{proposition}\cite[Corollary 3.8]{BCV24}
		\label{path ideal}
		For each $n\ge 3$, $l\ge 0$ and $r\ge 0$, the Betti numbers of $J_3(P_n)$ are
		\[
		\beta_{r,d}(J_3(P_n))=
		\begin{cases}
			\binom{n-3l-3}{r-l}\binom{n-2l-r-2}{2l-r+1}, & \text{ if }d=r+2l+3,\\
			0, &\text{otherwise.}
		\end{cases}
		\]
	\end{proposition}
	
	To compute the Betti numbers of $NI(P_n)$, we first compute the Betti numbers of the ideal $I_n$, defined as follows.
	
	\begin{definition}
		For a positive integer $n$, define the monomial ideal $I_n$ in the polynomial ring $R=\mathbb{K}[x_1,x_2,\dots,x_{n+1}]$ as
		\[
		\mathcal{G}(I_n) = \mathcal{G}(NI(P_{n+1})) \setminus \{ x_{n} x_{n+1} \}.
		\]
	\end{definition}
	
	We show that the ideals $I_n$ admit a Betti splitting, and using this fact, we obtain an explicit formula for their Betti numbers. Throughout, we adopt the convention that $\binom{0}{0}=1$ and $\binom{n}{-k}=0$ for $n,k>0$.

	\begin{proposition}
		\label{proposition:I_n=}
		For each $n\geq 1$, the Betti numbers can be written as follows:
		\begin{align}
			\label{eq:betti_I_n}
			\displaystyle
			\beta_{r,d}(I_n)=
			\begin{cases}
				\binom{n-2l-r-3}{2l-r+1}\binom{n-3l-3}{r-l}, & \text{ if }d=r+2l+3,\\
				\binom{n-2l-r-4}{2l-r+2}\binom{n-3l-5}{r-l-1}, & \text{ if }d=r+2l+4,\\
				0, & \text{otherwise.}
			\end{cases}
		\end{align}
		Here $l\ge -1$ is an integer.
	\end{proposition}
	
	\begin{proof}
		The proof is by induction on $n$. Observe that $\beta_{r,d}(I_1) = 0$ for all possible values of $r,d \geq 0$. Moreover, since $I_2=\langle x_1x_2 \rangle$, it is easy to see that 
		\begin{align}
        \label{eq:B_rd(x1x2)}
    		\beta_{r,d}(I_2) =
    		\begin{cases}
    			1, \quad &\text{if } (r,d) = (0,2),\\
    			0, \quad &\text{otherwise}.
    		\end{cases}
		\end{align}
		Thus, \cref{eq:betti_I_n} is valid for $n=1,2$. Now, assume that \cref{eq:betti_I_n} is valid for all values from $2$ to $n-1$, where $n\ge 3$ is an integer. We now proceed to verify it for $n$. In this case, if $(r,d)=(0,2)$, then it is easy to observe that $\beta_{0,2}=1$, and thus, \cref{eq:betti_I_n} is valid. Therefore, from now on wards, we may assume $(r,d)\neq (0,2)$. Let us now write the ideal $I_n$ as 
		\[
		I_n=J_n+K_n,
		\]
		where
		\[
		J_n = 
		\begin{cases}
			\langle 0 \rangle, \quad & \text{if } n = 1,\\
			\langle  x_1 x_2  \rangle, \quad & \text{if } n >1;
		\end{cases}
		\text{ and }
		K_n = 
		\begin{cases}
			\langle 0 \rangle, \quad & \text{if } n \leq 3,\\
			\langle x_2 x_3 x_4,\dots , x_{n-2} x_{n-1} x_n  \rangle, \quad & \text{if } n \geq 4.
		\end{cases}
		\]
		Observe that, for $n \geq 4$, $K_n$ is the $3$-path ideal of the path graph on the vertex set $\{x_2,\ldots,x_n\}$. Thus $K_n = J_3(P_{n-1})$ for $n \geq 4$. Since the ideal $J_n$ has a linear resolution, by Theorem \ref{theorem:betti_splitting} we see that $I_n = J_n + K_n$ is a Betti splitting, and hence, by Definition \ref{definition:betti_splitting},
		\[
		\beta_{r,d}(I_n)=\beta_{r,d}( J_n )+\beta_{r,d}(K_n)+\beta_{r-1,d}( J_n \cap K_n ).
		\]
		Note that,
		\begin{align*}
			J_n \cap K_n =
			\begin{cases}
				\langle 0 \rangle, \quad & \text{if } n < 3, \\
				\langle  x_1 x_2  \rangle . I_{n-2}, & \text{if } n \geq 3.
			\end{cases}
		\end{align*}
		Thus,
		\[
		\beta_{r-1,d} ( J_n \cap K_n ) =
		\begin{cases}
			0, \quad & \text{if } n < 3, \\
			\beta_{r-1,d-2}(I_{n-2}), & \text{if } n \geq 3.
		\end{cases}
		\]
		 Now, if $d=r+2l+3$, then $d-2=(r-1)+2(l-1)+4$. Therefore, using the induction hypothesis and \Cref{path ideal}, we obtain
		 \begin{align*}
		 	\beta_{r,d}(I_n)&=\binom{n-3l-4}{r-l}\binom{n-2l-r-3}{2l-r+1}+\binom{n-3l-4}{r-l-1}\binom{n-2l-r-3}{2l-r+1}\\
		 	&=\binom{n-2l-r-3}{2l-r+1}\binom{n-3l-3}{r-l},
		 \end{align*}
		 where the second equality follows from the binomial identity $\binom{a}{b-1}+\binom{a}{b}=\binom{a+1}{b}$. Finally, if we take $d=r+2l+4$, then $d-2=(r-1)+2l+3$. Hence, by \Cref{path ideal}, $\beta_{r,d}(J_3(P_{n-1}))=0$. Thus, by the induction hypothesis, we have 
		 \[
		 \beta_{r,d}(I_n)=\binom{n-2l-r-4}{2l-r+2}\binom{n-3l-5}{r-l-1},
		 \]
		 as desired.
	\end{proof}
	
	% \color{black}
	
	In the next theorem, we express the Betti numbers of $NI(P_n)$ in terms of the Betti numbers of $I_n$ obtained in \Cref{proposition:I_n=} above.

	\begin{theorem}
		\label{theorem:path_cases}
		Let $P_n$ be the path graph with $n$ vertices. Then, for 
		\(
		k= \left\lfloor \frac{n}{2} \right\rfloor
		\)
		and $n \geq 2$,
		\begin{equation}
			\begin{aligned}
				\label{eq:brd_path}
				\beta_{r,d}(NI(P_{n})) = &\sum \limits_{i=0}^{k-1} \left[ \beta_{r-i,d-2i}( \langle x_{n-2i-1} x_{n-2i} \rangle ) +  \beta_{r-i,d-2i} (I_{n-2i-1})\right] \\
				&+ \left( \frac{1 + (-1)^{n-1}}{2} \right) \beta_{r-k,d-2k}( \langle x_1 \rangle )
			\end{aligned}
		\end{equation}
	\end{theorem}
	
	\begin{proof}
		The proof is again by induction on $n$. If $n=2$, then $NI(P_2)=\l x_1x_2\r$, and thus \cref{eq:brd_path} is easily verified. For $n=3$, the minimal free resolution of $NI(P_3)$ is given by,
		\[
		0\longrightarrow R(-3)\xrightarrow{\begin{bmatrix}
				x_3\\
				-x_1
		\end{bmatrix}} R(-2)\oplus R(-2)\xrightarrow{\begin{bmatrix}
				x_1 x_2 & x_2 x_3
		\end{bmatrix}} NI(P_3)\longrightarrow 0.
		\]
		Therefore,
			\[
		\beta_{r,d} ( NI(P_3) ) =
		\begin{cases}
			2, \quad & \text{if } (r,d)=(0,2), \\
			1, \quad & \text{if } (r,d)=(1,3), \\
			0, \quad & \text{otherwise. }.
		\end{cases}
		\]
		 Thus, \cref{eq:brd_path} also holds for $n=3$. Now, assume that \cref{eq:brd_path} is valid for all values from $2$ to $n-1$, where $n\ge 4$ is an integer. We now proceed to verify it for $n$. Observe that the ideal $\langle  x_{n-1} x_n  \rangle$ has a linear resolution, and hence, by Theorem, \ref{theorem:betti_splitting}
		 \[
		 NI(P_n) = \langle  x_{n-1} x_n  \rangle + I_{n-1}
		 \]
		 is a Betti splitting. Note that, $ \langle  x_{n-1} x_n  \rangle \cap I_{n-1} = \langle ( x_{n-1} x_n ) \cdot NI(P_{n-2}) \rangle $, where $P_{n-2}$ is the induced path on the vertex set $\{x_1,\ldots,x_{n-2}\}$. Thus, by Definition \ref{definition:betti_splitting},
		 \begin{align}
		 	\label{eq:betti_splitting_path}
		 	\beta_{r,d}(NI(P_n)) = \beta_{r,d}( \langle  x_{n-1} x_n  \rangle ) + \beta_{r,d}(I_{n-1}) + \beta_{r-1,d-2}(NI(P_{n-2})).
		 \end{align}
		 By induction hypothesis, we have
		 \begin{equation*}
		 	\begin{aligned}
		 		\beta_{r-1,d-2}(NI(P_{n-2})) = &\sum \limits_{i=1}^{k-1} \left[ \beta_{r-i,d-2i}( \langle x_{n-2i-1} x_{n-2i} \rangle ) +  \beta_{r-i,d-2i} (I_{n-2i-1})\right] \\
		 		&+ \left( \frac{1 + (-1)^{n-1}}{2} \right) \beta_{r-k,d-2k}( \langle x_1 \rangle ).
		 	\end{aligned}
		 \end{equation*}
		 	Combining the above equation with \cref{eq:betti_splitting_path}, we get the desired result.
	\end{proof}
	
	% \color{black}
	
	Using \Cref{theorem:path_cases}, we now proceed to provide an explicit formula for the Betti numbers of $NI(P_n)$.

\begin{corollary}
	\label{corollary:path_betti}
	The values of $ \beta_{r,d}(NI(P_{n})) $ can be obtained using the following formula:
	\[
        {
		% \scriptsize 
        \footnotesize
		% \beta_{r,d}(NI(P_{n})) = 
		\begin{cases}
			\begin{aligned}
                &\binom{|k - r - 1|}{k - r - 1} \cdot \delta_{d,2r+2}  
				+ \binom{n - 2l - r - 4}{2l - r + 1} 
				\Bigg[
				\binom{n - 3l - 3}{r - l}
                -\binom{n - 3l - 4 - \left\lfloor \frac{q}{2} \right\rfloor }{r - l - 1 - \left\lfloor \frac{q}{2} \right\rfloor}\\
				&+ \binom{n - 3l - 4}{r - l -1} 
                - \binom{n - 3l - 3 - \left\lceil \frac{q}{2} \right\rceil }{r - l - \left\lfloor \frac{q}{2} \right\rfloor} \Bigg] 
                + \left( \frac{1 + (-1)^{n-1}}{2} \right) \delta_{r,k} \cdot \delta_{d,2r+1},
				\qquad\qquad \text{if } d = r + 2l + 3,
			\end{aligned}
			\\[10ex]
			\begin{aligned}
				&\binom{|k - r - 1|}{k - r - 1} \cdot \delta_{d,2r+2} 
				+ \binom{n - 2l - r - 5}{2l - r + 2} \Bigg[
				2 \binom{n - 3l - 5}{r - l - 1} 
                - \binom{n - 3l - 5 - \left\lfloor \frac{q}{2} \right\rfloor }{r - l - 1 - \left\lfloor \frac{q}{2} \right\rfloor}\\
				& - \binom{n - 3l - 5 - \left\lceil \frac{q}{2} \right\rceil }{r - l - 1 -  \left\lfloor \frac{q}{2} \right\rfloor}
				\Bigg]
                + \left( \frac{1 + (-1)^{n-1}}{2} \right) \delta_{r,k} \cdot \delta_{d,2r+1} ,
				\qquad\qquad\qquad\qquad\qquad \text{if } d = r + 2l + 4;
			\end{aligned}
		\end{cases}
        }
	\]
	where, $k = \lfloor\frac{n}{2}\rfloor,\ q = \min (k,r+1)$ and $\delta_{i,j}$ is the Kronecker delta function.
\end{corollary}

\begin{proof}
	We will explicitly find the values of each component in \cref{eq:brd_path}. By \cref{eq:B_rd(x1x2)} we can see that,
	\begin{align*}
		\sum \limits_{i=0}^{k-1} \beta_{r-i,d-2i}( \langle x_{n-2i-1} x_{n-2i} \rangle ) &= 
		\begin{cases}
			1, \quad \text{if } d = 2r + 2 \text{ and } r \leq k-1, \\
			0, \quad \text{otherwise}.
		\end{cases} \\
        &=  \binom{\mid k - r - 1 \mid}{k - r - 1} \cdot \delta_{d,2r+2}
	\end{align*}
	and $\beta_{r-k,d-2k}( \langle x_1 \rangle ) = \delta_{r,k} \cdot \delta_{d,2r+1}$. Notice that the terms in $\sum_{i=0}^{k-1} \beta_{r-i,d-2i}(I_{n-2i-1})$ shifts alternatively from the cases in \cref{eq:betti_I_n} and is $0$ when $i > r$. Thus, if $d=r+2l+3$, then we get
	\begin{align*}
        &\sum_{i=0}^{k-1} \beta_{r-i,d-2i}(I_{n-2i-1}) 
        % = \sum_{i=0}^{q-1} \beta_{r-i,d-2i}(I_{n-2i-1})
        \\
		=&\binom{n - 2l - r - 4}{2l - r + 1}
        \Bigg[
        \sum_{i=0}^{\lfloor\frac{q}{2}\rfloor - 1} \binom{n - 3l - 5 - i}{ r - l - 1 - i}
        + \sum_{i=0}^{\lceil\frac{q}{2}\rceil - 1} \binom{n - 3l - 4 - i}{ r - l - i}
        \Bigg]
        \\
		=& \binom{n - 2l - r - 4}{2l - r + 1} 
		\Bigg[
		\binom{n - 3l - 4}{r - l - 1} 
		- \binom{n - 3l - 4 - \left\lfloor \frac{q}{2} \right\rfloor}{r - l - 1 - \left\lfloor \frac{q}{2} \right\rfloor} \\
		&\qquad\qquad\qquad\qquad\qquad\quad
        +\binom{n - 3l - 3}{r - l}
        - \binom{n - 3l - 3 - \left\lceil \frac{q}{2} \right\rceil}{r - l - \left\lceil \frac{q}{2} \right\rceil}
		\Bigg],
	\end{align*}
	where the first equality is obtained by using the formula in \cref{eq:betti_I_n}, and the second equality is obtained using the binomial ideantity $\binom{a}{b-1}+\binom{a}{b}=\binom{a+1}{b}$. A similar calculation in the case $d=r+2l+4$, yields
	\begin{align*}
		\sum_{i = 0}^{k - 1} \beta_{r - i,d - 2i}(I_{n - 2i - 1})
		= &\binom{n - 2l - r - 5}{2l - r + 2} 
        \Bigg[
		2  \binom{n - 3l - 5}{r - l - 1} \\
        &- \binom{n - 3l - 5 - \left\lfloor \frac{q}{2} \right\rfloor}{r - l - 1 - \left\lfloor \frac{q}{2} \right\rfloor} 
        - \binom{n - 3l - 5 - \left\lceil \frac{q}{2} \right\rceil}{r - l - 1 - \left\lceil \frac{q}{2} \right\rceil}
        \Bigg].
	\end{align*}
	Summing up these formulas according to \cref{eq:brd_path} we obtain our result.
\end{proof}

\begin{example}
    %Using Theorem \ref{main theorem 1} we get the following Betti table of the graph $C_{12}(1,\widehat 2,3,4,5,6)$ in Figure \ref{figure 1}.
   Consider the path graph $P_{12}$. The Betti table of $NI(P_{12})$ computed using \Cref{corollary:path_betti} is given as follows.
    \begin{table}[h!]
    \centering
    \[
        \begin{array}{r|ccccccc}
              & 0 & 1 & 2 & 3 & 4 & 5 & 6 \\ \hline
            0 & 1 & \cdot & \cdot & \cdot & \cdot & \cdot & \cdot \\
            1 & \cdot & 2 & \cdot & \cdot & \cdot & \cdot & \cdot \\
            2 & \cdot & 8 & 10 & \cdot & \cdot & \cdot & \cdot \\
            3 & \cdot & \cdot & 12 & 14 & \cdot & \cdot & \cdot \\
            4 & \cdot & \cdot & 10 & 28 & 20 & \cdot & \cdot \\ 
            5 & \cdot & \cdot & \cdot & 6 & 16 & 12 & \cdot \\ 
            6 & \cdot & \cdot & \cdot & \cdot & \cdot & \cdot & 1 \\ \hline
            \text{Total} & 1 & 10 & 32 & 48 & 36 & 12 & 1
        \end{array}
    \]
    \caption{Betti table of $R/NI(P_{12})$}
    \end{table}
\end{example}
        
\section{Conclusion}\label{sec:conclusion}
	In the final part of this paper, we present some concluding remarks on the Betti numbers of closed neighborhood ideals of trees and examine the bridge-friendly condition for the closed neighborhood ideals of chordal and bipartite graphs.
		
	\subsection{Betti numbers for a class of trees} As shown in \Cref{thm:Bm of trees}, the closed neighborhood ideals of trees admit minimal BM resolutions and thus, in principle, by \Cref{critical cell betti pd}, one may compute their Betti numbers by counting the corresponding critical cells. In practice, however, this counting becomes tedious unless one can identify a correspondence between the critical cells of the original tree and those of certain subtrees. In \Cref{theorem:tree_betti}, we exploit such correspondences to express the Betti numbers of the closed neighborhood ideal of a class of trees in terms of the Betti numbers of appropriate subtrees.
		
	We begin with some basic results concerning how to verify the bridge, gap, and true gap conditions for certain special classes of monomial ideals. We remark that these results may be useful in future investigations of BM resolutions of arbitrary monomial ideals and their associated critical cells.

		\begin{lemma}
			\label{lemma:bridge_gap_I_I'}
			Let $I$ be a monomial ideal in $R=\mathbb{K}[x_1,x_2,\dots,x_n]$ and $A \subseteq \mathcal{G}(I)$ such that;
			\begin{enumerate}
				
				\item $\Su(\lcm(\mathcal{G}(I))) = \Su(\lcm(A)) \sqcup \Su(\lcm(\mathcal{G}(I')))$,
				
				\item For each $m \in A$, there exists some $x_i \mid m$ such that $x_i \nmid m'$ for all $m' \in \mathcal{G}(I) \setminus \{ m \}$,
				
			\end{enumerate}
			where $\mathcal{G}(I') = \mathcal{G}(I) \setminus A $. Then $\sigma \subseteq \mathcal{G}(I)$ has a bridge (resp. a gap) if and only if $\sigma' = \sigma \cap \mathcal{G}(I')$ has a bridge (resp. a gap). 
		\end{lemma}
		
		\begin{proof}
            Assume that $m_{1}$ is a bridge or a gap of $\sigma$ in $I$. Then, by \cite[Definition 2.3]{chau2024powers}, $\mathrm{lcm}(\sigma \cup \{ m_{1} \}) = \mathrm{lcm}(\sigma \setminus \{ m_{1} \})$. Since condition (2) asserts that no element of $A$ is a bridge or a gap in $I$, it follows that $m_{1} \notin A$. Thus, we can see that,
			% Note that, by condition $(2)$, none of the elements in $A$ can neither be a bridge nor be a gap in $I$. Assume $m_{1}$ is a bridge or a gap of $\sigma$ in $I$. By \cite[Definition 2.3]{chau2024powers}, $\lcm(\sigma \cup \{ m_{1} \}) = \lcm(\sigma \setminus \{ m_{1} \})$. \r Moreover, by the condition (2) above, $m_1\in \G(I')$\b. Thus, we can see that,
			\begin{align*}
				\lcm(\sigma \cup \{ m_{1} \}) &= \lcm((\sigma \cup \{ m_{1} \}) \cap \mathcal{G}(I')) \ \cdot \ \lcm((\sigma \cup \{ m_{1} \}) \cap A) \quad && \text{(by condition (1))} \\
				&= \lcm(\sigma' \cup \{ m_{1} \}) \ \cdot \ \lcm(\sigma \cap A) && \text{(since $m_1 \notin A$)}\\
				\lcm(\sigma \setminus \{ m_{1} \}) &= \lcm((\sigma \setminus \{ m_{1} \}) \cap \mathcal{G}(I')) \ \cdot\ \lcm((\sigma \setminus \{ m_{1} \}) \cap A) && \text{(by condition (1))} \\
				&= \lcm(\sigma' \setminus \{ m_{1} \}) \ \cdot \ \lcm(\sigma \cap A). && \text{(since $m_1 \notin A$)}
			\end{align*} 
			Hence, $ \lcm(\sigma' \cup \{ m_{1} \}) = \lcm(\sigma' \setminus \{ m_{1} \}) $. This implies $m_{1}$ is a bridge or a gap of $\sigma'$ in $I'$. Furthermore, since the implications used in this proof follows from the definitions of bridge and gap, one can see that the reverse implication also holds. Hence, the lemma is proved.
		\end{proof}
		
		As an application, we obtain the following two interesting corollaries.
		
		\begin{corollary}
			\label{corollary:truegap_I_I'}
			Let $I$ be a monomial ideal in $R=\mathbb{K}[x_1,x_2,\dots,x_n]$ and $A \subseteq \mathcal{G}(I)$ such that;
			\begin{enumerate}
				
				\item $\Su(\lcm(\mathcal{G}(I))) = \Su(\lcm(A)) \sqcup \Su(\lcm(\mathcal{G}(I')))$,
				
				\item For each $m \in A$, there exists some $x_i \mid m$ such that $x_i \nmid m'$ for all $m' \in \mathcal{G}(I) \setminus \{ m \}$,
				
			\end{enumerate}
			where $\mathcal{G}(I') = \mathcal{G}(I) \setminus A $. Then $m'$ is a true gap of $\sigma \subseteq \mathcal{G}(I)$ if and only if $m'$ is a true gap of $\sigma' = \sigma \cap \mathcal{G}(I')$.
		\end{corollary}
		
		\begin{proof}
			Let $m'$ be a true gap of $\sigma \subseteq \mathcal{G}(I)$. By \cref{lemma:bridge_gap_I_I'}, $m'$ is a gap of $\sigma' = \sigma \cap \mathcal{G}(I')$. Assume $m''$ is a bridge of $\sigma' \cup \{ m' \}$  and $m''<m'$. By \cref{lemma:bridge_gap_I_I'} again, $m''$ is a bridge of $\sigma \cup \{ m' \}$. Since $m'$ is a true gap of $\sigma$, $m''$ must be a bridge of $\sigma$. Then, using \cref{lemma:bridge_gap_I_I'}, we conclude that $m''$ is a bridge of $\sigma'$, which implies that $m'$ is a true gap of $\sigma'$. Similarly, one can easily prove the reverse direction. 
			%Conversely, suppose $m'$ is a true gap of $\sigma'$. By \cref{lemma:bridge_gap_I_I'}, $m'$ is also a gap of $\sigma$. Let $m''$ be a bridge of $ \sigma \cup \{ m' \} $ with $m'' < m'$. According to \cref{lemma:bridge_gap_I_I'}, $m''$ is a bridge of $\sigma' \cup \{ m' \}$. Since $m'$ is a true gap of $\sigma'$, it follows that $m''$ is also a bridge of $\sigma'$. Therefore, by \cref{lemma:bridge_gap_I_I'}, $m''$ is a bridge of $\sigma$. Thus, $m'$ is a true gap of $\sigma$. 
		\end{proof}
		
		\begin{corollary}
			\label{corollary:I_I'_critical}
			Let $I$ be a monomial ideal in $R = \mathbb{K}[x_1,\dots,x_n];\ A \subseteq \mathcal{G}(I)$ and $\mathcal{G}(I') = \mathcal{G}(I) \setminus A$ such that:
			\begin{enumerate}
				
				\item $\Su(\lcm(\mathcal{G}(I))) = \Su(\lcm(A)) \sqcup \Su(\lcm(\mathcal{G}(I')))$,
				
				\item For all $m \in A$, there exists $x_i \mid m$ and $x_i \nmid m'$ for all $m' \in \mathcal{G}(I) \setminus \{ m \}$,
				
				\item $I$ and $I'$ are bridge-friendly with respect to some total ordering of their minimal generators, where the ordering of the minimal generators of $I'$ is induced by that of the minimal generators of $I$.

			\end{enumerate}
			Then $\sigma \subseteq \mathcal{G}(I)$ is $I$-critical if and only if $\sigma' = \sigma \cap \mathcal{G}(I')$ is $I'$-critical.
		\end{corollary}
		
		\begin{proof}
			By \Cref{proposition:critical_bridgefriendly}, $\sigma$ is $I$-critical if and only if $\sigma$ has no bridge or true gap. Combining \cref{lemma:bridge_gap_I_I'} and \cref{corollary:truegap_I_I'}, we observe that this happens if and only if $\sigma'$ has no bridge or true gap. This completes the proof.
		\end{proof}
		We now present a method to express the Betti numbers of the closed neighborhood ideal of a class of trees in terms of the Betti numbers of appropriate subtrees.
		
		\begin{theorem}
			\label{theorem:tree_betti}
			Let $T$ be a tree with a leaf vertex $v_1$ and $N(v_1) = \{ v \}$ satisfying the following conditions:
			\begin{enumerate}
				
				\item $v$ has $n>0$ many leaf neighbors, and at most one non-leaf neighbor, say $u$,
				
				\item $u$ has at least one leaf neighbor (see \cref{fig:tree_special_case}).
				
			\end{enumerate}
			Then, for $T'=T\setminus \{v,v_1,\dots,v_n\}$ we have,
			\begin{equation*}
				% \label{eq:tree_betti_recursive}
				\beta_{r,d} \left( \frac{R}{NI(T)} \right) = \beta_{r,d} \left( \frac{R}{NI(T')} \right) + \sum\limits_{i=1}^{n} \binom{n}{i} \beta_{r-i,d-i-1} \left( \frac{R}{NI(T')} \right).
			\end{equation*}
		\end{theorem}
		
		\begin{proof}
			\begin{figure}[H]
				\begin{tikzpicture}
					% Define a common style for the nodes
					[scale=0.75, every node/.style={circle, draw, fill=blue!20, inner sep=0.8pt}]
					
					% Nodes with labels
					\node (v1) at (0,0) {$v_1$};
					\node (vn) at (2,0) {$v_n$};
					\node[inner sep=2.5pt] (v) at (1,1.5) {$v$};
					\node[inner sep=0.2pt] (u1) at (3,1.5) {\phantom{u1}};
					\node[inner sep=2.5pt] (u) at (2,3) {$u$};
					\node[inner sep=0.3pt] (x10) at (3,4.5) {$x_1^0$};
					\node[draw=none,fill=none] at (1,0) {$\dots$};

					% Edges
					\draw (v1) -- (v) -- (vn);
					\draw (v) -- (u) -- (u1);
					\draw[loosely dotted, line width=1pt] (u) -- (x10);
					
				\end{tikzpicture}
				\caption{Tree in \cref{theorem:tree_betti}}
				\label{fig:tree_special_case}
			\end{figure}
			Let $NI(T) = I,\ NI(T') = I'$, and 
            \[
            M_{r,d}(I) = \{ \sigma \subseteq \mathcal{G}(I) : \sigma \text{ is a critical subset},\ |\sigma| = r,\ \deg(\lcm(\sigma)) = d \}.
            \]
            We also define $M_{r,d}(I')$ analogously. Let $A=\{vv_1,\ldots,vv_n\}$. 
			By \Cref{thm:Bm of trees} and \Cref{critical cell betti pd}, we know that $\beta_{r,d} \left( R/I \right) = |M_{r,d}(I)|$ and $\beta_{r,d} \left( R/I' \right) = |M_{r,d}(I')|$. Now, observe that if $m \in \G(I)$ and $v \mid m $, then $m = v v_i$ for some $i \in \{ 1, \dots , n\}$. Furthermore, by \cref{corollary:I_I'_critical}, if $\sigma \in M_{r,d}(I)$ and $v \nmid m $ for all $m \in \sigma$, then \( \sigma \in M_{r,d}(I') \) . 
			Next, we focus on characterizing the critical cells $\sigma\in M_{r,d}(I)$ such that $v\mid m$ for some $m\in\sigma$; and we do this based on $|\sigma_v|$, where  $\sigma_v = \{ m \in \sigma : v \mid m \}$. First, consider a $\sigma \in M_{r,d}(I)$ with $|\sigma_v|=1$. This means there exists exactly one $v v_i$ in $\sigma$ for some $i \in \{ 1, \dots , n\}$. By \cref{corollary:I_I'_critical}, $\sigma \setminus \{v v_i\} \in M_{r-1,d-2}(I')$. There exist exactly $n$ subsets of $\mathcal{G}(I)$ in $M_{r,d}(I)$ with $|\sigma_v| = 1$, each corresponding to a distinct choice of $v_i$ for $i \in \{ 1, \dots , n \}$. In general, for each $i \in \{1, \dots, n\}$ and $\sigma' \in M_{r-i,d-i-1}(I')$, there are exactly $\binom{n}{i}$ subsets of $M_{r,d}(I)$ with $\sigma \setminus \sigma_v = \sigma'$. Therefore, using \cref{corollary:I_I'_critical} again we obtain the required formula:
            \begin{equation*}
				% \label{eq:tree_betti_recursive}
				\beta_{r,d} \left( \frac{R}{NI(T)} \right) = \beta_{r,d} \left( \frac{R}{NI(T')} \right) + \sum\limits_{i=1}^{n} \binom{n}{i} \beta_{r-i,d-i-1} \left( \frac{R}{NI(T')} \right).
			\end{equation*}
		\end{proof}
		
		In \cref{fig:tree_special_case}, the crucial feature is that $u$ has a leaf child. Otherwise, $m(T,u) \in \mathcal{G}(NI(T))$ and $v \in \Su(m(T,u))$, in which case Corollary \ref{corollary:I_I'_critical} would no longer be applicable.
		For example, for the path graph $P_7$,
		\[
		4 = \beta_{3,6} \left( \frac{R}{NI(P_7)} \right) \neq \beta_{3,6} \left( \frac{R}{NI(P_5)} \right) +  \beta_{2,4} \left( \frac{R}{NI(P_5)} \right) = 0 + 3.
		\]
		
		This observation shows that our current approach does not extend \Cref{theorem:tree_betti} to all trees. Consequently, the following natural and compelling problem remains open.
		
		\begin{question}
			Is there a recursive formula for the Betti numbers of the closed neighborhood ideals of trees?
		\end{question}

		\subsection{Bridge-friendly condition for chordal and bipartite graphs}Two well-known and important generalisations of trees are the classes of chordal graphs and bipartite graphs. In view of \Cref{thm:Bm of trees}, it is natural to ask whether the closed neighborhood ideals of chordal graphs and bipartite graphs also satisfy the bridge-friendly condition.
		
		For closed neighborhood ideals of chordal graphs, our computations with SageMath\cite{Sage} show that every chordal graph with $8$ or fewer vertices admits a bridge-friendly ordering of the generators. On the other hand, the graph displayed in \cref{fig:nonbridgefriendly_chordal} is the only chordal graph with nine vertices for which no ordering of $NI(G)$ is bridge-friendly. In fact, one can show directly in the following way that for any ordering of the minimal generators of $\G(NI(G))$, the ideal is not bridge-friendly:
		
		\begin{figure}[h]
			\centering
			\begin{tikzpicture}
				% Define a common style for the nodes
				[scale=0.70,every node/.style={circle, draw, fill=blue!20, inner sep=0.8pt}]
				
				% Nodes with labels
				\node (x7) at (0.5,-1) {$x_7$};
				\node (x5) at (2,0) {$x_5$};
				\node (x4) at (4,0) {$x_4$};
				\node (x6) at (6,0) {$x_6$};
				\node (x8) at (7.5,-1) {$x_8$};
				\node (x2) at (3,1.5) {$x_2$};
				\node (x3) at (5,1.5) {$x_3$};
				\node (x1) at (4,3) {$x_1$};
				\node (x0) at (4,4.5) {$x_0$};
				
				% Edges
				\draw (x7) -- (x5) -- (x4) -- (x6) -- (x8);
				\draw (x5) -- (x2) -- (x1) -- (x0);
				\draw (x6) -- (x3) -- (x1);
				\draw (x2) -- (x3) -- (x4) -- (x2);
				
			\end{tikzpicture}
			\caption{Non Bridge Friendly Chordal Graph}
			\label{fig:nonbridgefriendly_chordal}
		\end{figure}
		
		We have
		\[
		\G(NI(G))=\{x_0x_1,x_5x_7,x_6x_8,x_1x_2x_3x_4x_5,x_1x_2x_3x_4x_6,x_2x_3x_4x_5x_6\}.
		\]
		Let $>$ be an arbitrarily chosen but fixed ordering on $\G(NI(G))$. Without loss of generality, let $x_1x_2x_3x_4x_5>x_1x_2x_3x_4x_6>x_2x_3x_4x_5x_6$. Take 
		\begin{align*}
			\sigma&=\{x_0x_1,x_1x_2x_3x_4x_6,x_2x_3x_4x_5x_6\},\\
			\tau&=\{x_0x_1,x_1x_2x_3x_4x_5,x_2x_3x_4x_5x_6\}.
		\end{align*}
		Observe that the generators $x_5x_7$ and $x_6x_8$ can never be a gap of either $\sigma$ or $\tau$. Furthermore, $x_1x_2x_3x_4x_6$ (respectively, $x_1x_2x_3x_4x_5$) is a bridge of $\sigma$ (respectively, $\tau$). Now, for $\sigma$, the monomial $x_1x_2x_3x_4x_5$ is a gap. However, since $x_1x_2x_3x_4x_5>x_1x_2x_3x_4x_6$, we see that $\sigma$ is a potentially type-2 subset of $\G(NI(G))$. As for $\tau$, we observe that $x_1x_2x_3x_4x_6$ is a gap of $\tau$, and $x_1x_2x_3x_4x_5>x_1x_2x_3x_4x_6$. We proceed to show that $x_1x_2x_3x_4x_6$ is not a true gap of $\tau$. Indeed, $\tau\cup\{x_1x_2x_3x_4x_6\}$ has a new bridge $x_2x_3x_4x_5x_6$ dominated by $x_1x_2x_3x_4x_6$, and thus, $x_1x_2x_3x_4x_6$ cannot be a true gap of $\tau$. Therefore, $\tau$ is a potentially type-2 subset of $\G(NI(G))$. Now observe that
		\[
		\sigma\setminus\mathrm{sb}(\sigma)=\tau\setminus\mathrm{sb}(\tau)=\{x_0x_1, x_2x_3x_4x_5x_6\},
		\]
		where $\mathrm{sb}(\sigma)$ (respectively, $\mathrm{sb}(\tau)$) denotes the smallest bridge of $\sigma$ (respectively, $\tau$). Thus, $\tau$ is a potentially type-2 subset of $\G(NI(G))$ which is not type-2. Hence, $NI(G)$ is not a bridge-friendly ideal.
		
		The previous example demonstrates that not all chordal graphs are bridge-friendly. This naturally leads to the following question.
		
		\begin{question}
			Classify all chordal graphs that are bridge-friendly.
		\end{question}
	
        In case of bipartite graphs, computations in SageMath \cite{Sage} shows that the closed neighborhood of the graph in Figure \ref{fig:bipartite} is not bridge-friendly for any ordering of its generators. Therefore, the same question could be asked for bipartite graphs as well. 
			
			\begin{figure}[h]
				\centering
				\begin{tikzpicture}
					% Define a common style for the nodes
					[scale=0.70,every node/.style={circle, draw, fill=blue!20, inner sep=0.8pt}]
					
					% Nodes with labels
					\node (x5) at (4,2) {$x_5$};
					\node (x4) at (2,2) {$x_4$};
					\node (x3) at (0,2) {$x_3$};
					\node (x2) at (4,0) {$x_2$};
					\node (x1) at (2,0) {$x_1$};
					\node (x0) at (0,0) {$x_0$};
					
					% Edges
					\draw (x0) -- (x1) -- (x4) -- (x3) -- (x0);
					\draw (x1) -- (x2) -- (x5) -- (x4);
					
				\end{tikzpicture}
				\caption{Non Bridge Friendly Bipartite Graph}
				\label{fig:bipartite}
			\end{figure}

		For Cycle Graphs, SageMath\cite{Sage} computations show that the closed neighborhood ideal of $C_{5}$ is not bridge-friendly for any ordering of its generators.  Furthermore, we know that the closed neighborhood ideal of a cycle coincides with its $3$-path ideal, and by \cite[Theorem~46]{chau2024path}, this ideal admits a minimal generalised BM resolution. Hence, it would be interesting to determine whether the closed neighborhood ideal of a general bipartite graph always admits a minimal generalised BM resolution.
		
		\subsection{Regularity of the closed neighborhood ideals of trees:} Recall that the Castelnuovo-Mumford regularity (in short, the \textit{regularity}), denoted by $\reg(R/I)$, is defined as \[ \reg(R/I) = max\{ d-r : \beta_{r,d}(R/I) \neq 0\text{ for some }r,d\}. \]
       According to \cite[Theorem 3.4]{CJRS2025}, the regularity of the closed neighborhood ideal of a tree is equal to the matching number of the tree. Moreover, it follows from \Cref{critical cell betti pd} that 
		\[
		\reg \left( \frac{R}{NI(G)} \right) = \max 
		\left\{ 
		d - r \;:\;
		\begin{aligned}
			&\sigma \subseteq \G(NI(G)) \text{ is critical} ,\\
			&|\sigma| = r,\ \deg(\lcm(\sigma)) = d
		\end{aligned}
		\right\}.
		\]
		Thus, it would be of interest to obtain a combinatorial proof of the fact that the regularity of the closed neighborhood ideal of a tree is the same as its matching number, leveraging the descriptions of the critical cells, analogous to the results presented in \Cref{pd tree section}.
		\section*{Acknowledgements}
		Part of this project was initiated during the first author's visit to the second author at the Chennai Mathematical Institute. The authors sincerely thank Priyavrat Deshpande for facilitating this visit and the Chennai Mathematical Institute for the excellent working environment.
		Ajay P. Joseph would like to express his gratitude to the National Institute of Technology Karnataka for the Doctoral Research Fellowship. Amit Roy was supported by a Postdoctoral Fellowship as well as a grant from the Infosys Foundation during his stay at the Chennai Mathematical Institute. Anurag Singh is supported by the NBHM (National Board for Higher Mathematics) project grant 02011/25/2025/NBHM(R.P.)/9827.
		
		\subsection*{Data availability statement} Data sharing is not applicable to this article as no new data were created or
		analyzed in this study.
		
		\subsection*{Conflict of interest} The authors declare that they have no known competing financial interests or personal
		relationships that could have appeared to influence the work reported in this paper.

		\bibliographystyle{abbrv}
		\bibliography{ref}
		
	\end{document}